\documentclass[a4paper,reqno]{amsart}

\usepackage{amssymb}
\usepackage{amsmath}
\usepackage{amsthm}
\usepackage{enumitem}
\usepackage{xcolor}
\usepackage{graphicx}
\usepackage{pgfarrows,pgfnodes}
\usepackage{url}
\usepackage{textcomp}
\usepackage{ytableau}
\usepackage{dsfont}
\usepackage{tikz}

\usepackage{fullpage}

\usepackage{todonotes}

\usepackage{texdraw,amstext,amsfonts,color,tabu,stmaryrd}

\usepackage{mathtools,tabularx}

\usepackage{float}

\definecolor{foge}{rgb}{0.1, 0.6, 0.1}

\newcommand{\m}{\medbreak}
\newcommand{\bi}{\bigbreak}

\newcommand{\Thm}[1]{Theorem \ref{#1}}

\newcommand{\Lem}[1]{Lemma \ref{#1}}

\newcommand{\Sct}[1]{Section \ref{#1}}
\newcommand{\Def}[1]{Definition \ref{#1}}

\newcommand{\I}{\{0,\dots,n-1\}}

\newcommand{\iso} {\buildrel \sim \over \rightarrow}

\newcommand{\Ll}{\Lambda}
\newcommand{\Z}{\mathbb{Z}}
\newcommand{\B}{\mathcal{B}}

\newcommand{\Bb}{\mathbb{B}}

\newcommand{\wt}{\overline{\mathrm{wt}}}
\newcommand{\C}{\mathcal{C}}

\newcommand{\E}{\mathcal{E}}

\newcommand{\Co}{\mathcal{C}}

\newcommand{\Pp}{\mathcal{P}}

\newcommand{\Ppp}{\Pp^{\gg}_{\co}}
\newcommand{\Psuc}{\Pp^{\succ}_{\co}}
\newcommand{\Ppm}{\Pp^{\gg}_{\cmu}}
\newcommand{\Pppm}{\Pp^{\gtrdot}_{\cmu}}
\newcommand{\Psucm}{\Pp^{\succ}_{\cmu}}

\newcommand{\Pppp}{\Pp^{\gtrdot}_{\co}}

\newcommand{\od}{\geq}

\newcommand{\gf}{\mathfrak g}
\newcommand{\h}{\mathfrak h}
\newcommand{\ot}{\otimes}
\newcommand{\p}{\mathfrak p}

\newcommand{\co}{c_{g}}
\newcommand{\cmu}{c_{g_0}\cdots c_{g_{t-1}}}
\newcommand{\ov}{\overline}

\newcommand{\ssss}{\{1,\ldots,s-1\}}

\numberwithin{equation}{section}

\newtheorem{theo}{Theorem}[section]

\newtheorem{prop}[theo]{Proposition}
\newtheorem{lem}[theo]{Lemma}
\newtheorem{cor}[theo]{Corollary}
\newtheorem{rem}[theo]{Remark}
\newtheorem{ex}[theo]{Example}

\theoremstyle{definition} 
\newtheorem{deff}[theo]{Definition}

\newtheorem{manualtheoreminner}{Theorem}
\newenvironment{manualtheorem}[1]{%
  \renewcommand\themanualtheoreminner{#1}%
  \manualtheoreminner
}{\endmanualtheoreminner}

\title{Multi-grounded partitions and character formulas}
\author{Jehanne Dousse and Isaac Konan}

\begin{document}

\begin{abstract}
We introduce a new generalisation of partitions, multi-grounded partitions,
related to ground state paths indexed by dominant weights of Lie algebras.
We use these to express characters of irreducible highest weight
modules of Kac--Moody algebras of affine type as generating functions for
multi-grounded partitions.
This generalises the approach of our previous paper, where
only irreducible highest weight modules with constant ground
state paths were considered, to all ground state paths.
As an application, we compute the characters of the level $1$ modules of the
affine Lie algebras $A_{2n}^{(2)}(n\geq 2)$, $D_{n+1}^{(2)}(n\geq 2)$, $A_{2n-1}^{(1)}(n\geq 3)$, $B_{n}^{(1)}(n\geq 3)$, and $D_{n}^{(1)}(n\geq 4)$.
\end{abstract}

\maketitle

\section{Introduction and statement of results}
Let $\gf$ be a Kac--Moody affine Lie algebra and let $\h^*$ be the dual of its Cartan subalgebra. Let $P^+$ be the set of dominant integral weights, and $L(\lambda)$ an irreducible highest weight $\gf$-module of highest weight $\lambda \in P^+$.
Then the character of $L(\lambda)$ is defined as
$$
\mathrm{ch} (L(\lambda)) = \sum_{\mu\in \h^*} \dim L(\lambda)_{\mu} \cdot e^{\mu},
$$
where $e$ is a formal exponential, and $\dim L(\lambda)_{\mu}$ is the dimension of the weight space $L(\lambda)_{\mu}$ in the weight space decomposition
$$L(\lambda)=\bigoplus_{\mu \in P} L(\lambda)_{\mu}.$$
More background and definitions can be found in Section \ref{sec:lie}.

Character formulas have been widely studied, starting with the famous Weyl--Kac character formula \cite{Kac}:
\begin{equation}
\label{eq:weylkac}
\mathrm{ch} (L(\lambda)) = \frac {\sum _{ w \in W } \mathrm{sgn}(w) e ^ { w ( \lambda + \rho ) - \rho } } { \prod _ { \alpha \in \Delta ^ { + }} ( 1 - e ^ { -\alpha } ) ^ { \text{dim} \mathfrak{g} _ { \alpha }  } },
\end{equation}
where $W$ is the Weyl group of $\gf$, $\Delta^{+}$ the set of positive roots of $\gf$, $\mathrm{sgn}(w)$ the signature of $w$, $\rho \in \h^*$ the Weyl vector, and  
$\mathfrak{g} _ { \alpha }$ the $\alpha$ root space of $\gf$.

Equation \eqref{eq:weylkac} is beautiful, but it is not so well suited to compute characters in practice. Moreover, even though by definition $e^{-\lambda}\mathrm{ch} (L(\lambda))$ is a series with positive coefficients in the $e^{-\alpha_i}$'s, this positivity is not explicit from the formulation in \eqref{eq:weylkac}. We now briefly explain what solutions have been given to work around these issues, and present a new method which allows us to give simple non-specialised character formulas using perfect crystals and a new generalisation of integer partitions.

\medskip
The first solution to obtain simple character formulas is to perform certain specialisations, i.e. for each of the simple roots $\alpha_i$, applying the transformations $e^{-\alpha_i} \rightarrow q^{s_i}$ for some integer $s_i$. Using this method, it is possible to transform the Weyl--Kac character formula into infinite products. From this point of view, the most effective specialisation is the principal specialisation, where $e^{-\alpha_i} \rightarrow q$ for all $i$. It has been widely exploited in the theory of partition identities related to representations of affine Lie algebras, see for example \cite{Capparelli,GOW16,Meurman,Meurman2,Meurman3,Nandi,Primc1,PrimcSikic,Siladic}. Lepowsky and Milne \cite{Le-Mi1,Le-Mi2} were the first to expose the connection by noting that up to the $(q;q^2)_{\infty}$ factor, the principal specialisation of the Weyl--Kac character formula for level $3$ standard modules of the affine Lie algebra $A_1^{(1)}$ is the product side of the Rogers--Ramanujan identities:
\begin{align*}
\sum_{n\geq 0}\frac{q^{n^2}}{(q;q)_n}&=\frac{1}{(q;q^5)_{\infty} (q^4;q^5)_{\infty}},\\
\sum_{n\geq 0}\frac{q^{n^2+n}}{(q;q)_n}&=\frac{1}{(q^2;q^5)_{\infty} (q^3;q^5)_{\infty}}.
\end{align*} 
Here and in the whole paper, we use the standard $q$-series notation: for $n \in \mathbb{N} \cup \{\infty\}$ and $j \in \mathbb{N}$,
\begin{align*}
(a;q)_n &:= \prod_{k=0}^{n-1} (1-aq^k),\\
(a_1, \dots, a_j ; q)_n &:= (a_1;q)_n \cdots (a_j;q)_n.
\end{align*}
Lepowsky and Wilson \cite{Lepowsky,Lepowsky2} later gave an interpretation of the sum side by constructing a basis of these standard modules using vertex operators. Their method has then led to the discovery of many new $q$-series and partition identities, see e.g. \cite{Capparelli,Nandi,PrimcSikic,Siladic}.

However, without performing a specialisation, it is in general difficult to reduce the Weyl--Kac character formula to obtain a combinatorially simple character formula, with perhaps the exception of the Kac--Peterson formulas \cite{KacPeterson}, which still required a lot work using modular forms.

On the other hand, Bartlett and Warnaar \cite{BW15} gave non-specialised formulas with explicitly positive coefficients for the characters of certain highest weight modules of the affine Lie algebras $C_n^{(1)}$, $A_{2n}^{(2)}$, and $D_{n+1}^{(2)}$ as sums using Hall--Littlewood polynomials. This led them to generalisations for the Macdonald identities for $B_n^{(1)},$ $C_n^{(1)},$ $A_{2n-1}^{(2)},$ $A_{2n}^{(2)}$, and $D_{n+1}^{(2)}.$ Using Macdonald--Koornwinder theory, Rains and Warnaar \cite{WarnaarRains} found additional character formulas for these Lie algebras, together with new Rogers--Ramanujan type identities.

\medskip
In a different direction, Kang, Kashiwara, Misra, Miwa, Nakashima, and Nakayashiki \cite{KMN2a,KMN2b} introduced the theory of perfect crystals to study the irreducible highest weight modules over quantum affine algebras. The proved the so-called ``(KMN)$^2$ crystal base character formula'' \cite{KMN2a},
$$\mathrm{ch}(L(\lambda)) = \sum_{\p \in \mathcal P(\lambda)}  e^{\mathrm{wt} \p},$$
which expresses the character $\mathrm{ch}(L(\lambda))$ as series indexed by $\lambda$-paths. Here the weight $\mathrm{wt} \p$ is computed using the energy function of a perfect crystal. More detail will be given in Section \ref{sec:crystal}.

Primc \cite{Primc} was the first to use this character formula to give new Rogers--Ramanujan type identities related to the level $1$ standard modules of $A_1^{(1)}$ and $A_2^{(1)}$. The product side came from the principally specialised Weyl--Kac character formula again, while the sum side came from the principally specialised (KMN)$^2$ crystal base character formula.
In a couple of previous papers \cite{DK19,DK19-2}, the authors generalised Primc's identities to $A_{n-1}^{(1)}$ for all $n$, and managed to avoid doing a specialisation, therefore retrieving the Kac--Peterson character formula with all its parameters. To do this, we established a bijection between $\lambda$-paths and a new generalisation of partitions called ``grounded partitions'', in the case where the ground state path is constant.

\medskip
Recall that a partition $\pi$ of a positive integer $n$ is a non-increasing sequence of natural numbers $(\pi_1,\dots,\pi_s)$, called parts, whose sum is $n$, the partitions of $4$ being $(4),\ (3,1),\ (2,2),\ (2,1,1),$ and $(1,1,1,1).$
Grounded partitions, which are defined more rigorously in Section \ref{sec:groundedpartitions}, are partitions whose smallest part is fixed, where all the parts are coloured and satisfy particular difference conditions.
In \cite{DK19-2}, using a bijection and the (KMN)$^2$ crystal base character formula, we obtained new character formulas expressing the characters directly as generating functions for grounded partitions:
\begin{align}
\sum_{\pi\in \Pppp} C(\pi)q^{|\pi|} &= e^{-\lambda}\mathrm{ch}(L(\lambda)), \nonumber\\
\sum_{\pi\in \Ppp} C(\pi)q^{|\pi|} &= \frac{e^{-\lambda}\mathrm{ch}(L(\lambda))}{(q;q)_{\infty}}, \label{eq:groundchar}
\end{align}
where $\Pppp$ and $\Ppp$ are sets of grounded partitions depending on the module considered and on the energy function of the corresponding crystal, $q= e^{-\delta/d_0}$ where $\delta = d_0 \alpha_0 + d_1 \alpha_1 + \cdots + d_{n-1} \alpha_{n-1}$ is the null root, and $C(\pi)$ is the colour sequence of the grounded partition $\pi$. This character formula is stated more rigorously in Theorem \ref{theo:formchar}.

However useful, this formula only applies for standard modules whose ground state paths are constant, which is not the case of most modules. The goal of this paper is to extend our method to treat all standard modules, whatever their ground state paths are. To do so, we extend our definition of grounded partitions and introduce so-called ``multi-grounded partitions''. We define them rigorously and prove their connection with crystals and characters in Section \ref{sec:multiground}, but let us already say that among other conditions, multi-grounded partitions now have their $t$ smallest parts fixed for some $t \geq 1$.

Again, by establishing a bijection with $\lambda$-paths, we transform the (KMN)$^2$ crystal base character formula into a character formula using generating functions on multi-grounded partitions (all the notations will become clear once the reader gets to Section \ref{sec:multiground}).
\begin{theo} \label{theo:formchar2}
Setting $q=e^{-\delta/(d_0D)}$ and $c_b=e^{\wt b}$ for all $b\in \B$, we have $c_{g_0}\cdots c_{g_{t-1}}=1$, and the character of the irreducible highest weight $U_q(\gf)$-module
$L(\lambda)$ is given by the following expressions:
\begin{align*}
\sum_{\mu\in _t\Pppm} C(\pi)q^{|\pi|} &= e^{-\lambda}\mathrm{ch}(L(\lambda)),\\
 \sum_{\pi\in \, _t^d\Ppm} C(\pi)q^{|\pi|} &= \frac{e^{-\lambda}\mathrm{ch}(L(\lambda))}{(q^d;q^d)_{\infty}}.
\end{align*}
\end{theo}
One big advantage of Theorem \ref{theo:formchar2} is that one does not need to perform a specialisation, and that, being generating functions for combinatorial objects, the series always have obviously positive coefficients. Moreover, as we will see in Section \ref{sec:examples}, these generating functions are relatively easy to compute in practice.

As examples of application, we use our new method and Theorem \ref{theo:formchar2} to compute character formulas for irreducible highest weight level one modules of classical Lie algebras:
\begin{itemize}
\item $\Lambda_0$ for $A_{2n}^{(2)}(n\geq 2)$,
\item $\Lambda_0$ and $\Lambda_n$ for $D_{n+1}^{(2)}(n\geq 2)$,
\item $\Lambda_0, \Lambda_1$ for $A_{2n-1}^{(1)}(n\geq 3)$,
\item $\Lambda_0$ and $\Lambda_1$ for $B_{n}^{(1)}(n\geq 3)$,
\item $\Lambda_0, \Lambda_1, \Lambda_{n-1},\Lambda_n$ for $D_{n}^{(1)}(n\geq 4)$.
\end{itemize}
All these formulas are non-specialised, with obviously positive coefficients, and are either infinite products or sums of two infinite products. We restrict ourselves to examples of level $1$ in this paper for brevity, but the method applies in theory to any level.

The first two formulas were already proved by Frenkel and Kac \cite{FrenkelKac} by  constructing basic representations using vertex operators from the dual resonance theory in physics.
These identities were reproved by the second author in \cite{Konan_ww2}  using another method based on a generalisation of Glaisher's identity \cite{Glaisher} through a Sylvester-style bijection.
Nonetheless, we reprove them here to illustrate that our new method gives very simple proofs.

\begin{theo}[Frenkel--Kac]
\label{theo:chara2n2}
Let $n\geq 2$, and let $\Lambda_0, \dots, \Lambda_{n}$ be the fundamental weights  and $\alpha_0, \dots, \alpha_{n}$ be the simple roots of $A_{2n}^{(2)}$. Let $ \delta= 2\alpha_0 + \cdots + 2\alpha_{n-1}+\alpha_n$ be the null root.
Let us set
$$q=e^{-\delta/2} \quad \text{and} \quad c_i = e^{\alpha_i + \cdots + \alpha_{n-1} +\alpha_n/2} \text{ for all } i \in \{1,\dots , n \}.$$
We have
$$e^{-\Ll_0}\mathrm{ch}(L(\Lambda_0))= \prod_{k=1}^n (-c_k q;q^2)_{\infty} (-c_k^{-1}q;q^2)_{\infty}.$$
\end{theo}

\begin{theo}[Frenkel--Kac]
\label{theo:chardn2}
Let $n\geq 2$, and let $\Lambda_0, \dots, \Lambda_{n}$ be the fundamental weights  and $\alpha_0, \dots, \alpha_{n}$ be the simple roots of $D_{n+1}^{(2)}$.
Let $\delta= \alpha_0+\cdots+\alpha_n$ be the null root.
Let us set
$$q=e^{-\delta} \quad \text{and} \quad c_i = e^{\alpha_i + \cdots +\alpha_n} \text{ for all } i \in \{1,\dots , n \}.$$
We have 
\begin{align}
e^{-\Ll_0}\mathrm{ch}(L(\Ll_0)) &= \frac{1}{(q;q^2)_{\infty}} \prod_{k=1}^n (-c_k q;q^2)_{\infty} (-c_k^{-1}q;q^2)_{\infty}, \label{eq:d21}
\\e^{-\Ll_n}\mathrm{ch}(L(\Ll_n)) &= \frac{1}{(q;q^2)_{\infty}} \prod_{k=1}^n (-c_k q^2;q^2)_{\infty} (-c_k^{-1};q^2)_{\infty}. \label{eq:d22}
\end{align}
\end{theo}

We now turn to character formulas which, to our knowledge, are new. They rely on the parity of the number of parts, we therefore introduce some notation before stating them.

Let $G=G(x_1, \dots, x_n)$ be a power series in several variables $x_1, \dots, x_n.$ For $k \leq n$, we denote by $\E_{x_1, \dots , x_k} (G)$ the sub-series of $G$ where we only keep the terms in which the sum of the powers of $x_1, \dots , x_k$ is even.
Note that if $G$ has only positive coefficients, then the same is true for $\E_{x_1, \dots , x_k} (G)$ for all $k$. There is a simple formula to obtain $\E_{x_1, \dots , x_k} (G)$ from $G$:

\begin{equation}
\label{eq:evennumberparts}
\E_{x_1, \dots , x_k} (G)= \frac{1}{2} \Big(G(x_1, \dots, x_k, x_{k+1}, \dots , x_n) + G(-x_1, \dots, -x_k, x_{k+1}, \dots , x_n)\Big).
\end{equation}

We can now state our character formulas in a simple form.

\begin{theo}\label{theo:chara2n1}
Let $n\geq 3$, and let $\Lambda_0, \dots, \Lambda_{n}$ be the fundamental weights  and $\alpha_0, \dots, \alpha_{n}$ be the simple roots of $A_{2n-1}^{(2)}$.
Let $\delta= \alpha_0+\alpha_1+2\alpha_2\cdots+2\alpha_{n-1}+\alpha_n$ be the null root.
Let us set
$$q=e^{-\delta/2} \quad \text{and} \quad c_i = e^{\alpha_i + \cdots + \alpha_{n-1} +\alpha_n/2} \text{ for all } i \in \{1,\dots , n \}.$$
We have
\begin{align}
e^{-\Ll_0}\mathrm{ch}(L(\Ll_0))&= \E_{c_1, \dots , c_n} \left( (q^2;q^4)_{\infty} \prod_{k=1}^n (-c_k q;q^2)_{\infty} (-c_k^{-1}q;q^2)_{\infty}\right)\nonumber \\ 
&= \frac{(q^2;q^4)_{\infty}}{2} \left(\prod_{k=1}^n (-c_k q;q^2)_{\infty} (-c_k^{-1}q;q^2)_{\infty}+\prod_{k=1}^n (c_k q;q^2)_{\infty} (c_k^{-1}q;q^2)_{\infty}\right), \label{eq:a210}
\\e^{-\Ll_1}\mathrm{ch}(L(\Ll_1))&= \E_{c_1, \dots , c_n} \left( (q^2;q^4)_{\infty} (-c_1 q^3;q^2)_{\infty} (-c_1^{-1}q^{-1};q^2)_{\infty} \prod_{k=2}^n (-c_k q;q^2)_{\infty} (-c_k^{-1}q;q^2)_{\infty}\right) \nonumber\\
&= \frac{(q^2;q^4)_{\infty}}{2} \left((-c_1 q^3;q^2)_{\infty} (-c_1^{-1}q^{-1};q^2)_{\infty} \prod_{k=2}^n (-c_k q;q^2)_{\infty} (-c_k^{-1}q;q^2)_{\infty} \right. \label{eq:a211}\\
& \left.\qquad \qquad \qquad +(c_1 q^3;q^2)_{\infty} (c_1^{-1}q^{-1};q^2)_{\infty} \prod_{k=2}^n (c_k q;q^2)_{\infty} (c_k^{-1}q;q^2)_{\infty}\right). \nonumber
\end{align}
\end{theo}

The next theorem concerns the Lie algebra $B_{n}^{(1)}$. Note that the second author proved a character formula for $L(\Ll_n)$, another level $1$ module, in \cite{Konan_ww2}. However we do not reprove it here as it can be easily proved using the character formula \eqref{eq:groundchar} of \cite{DK19-2} and does not need any of the innovations of the current paper. However, the character formulas for the modules $L(\Ll_0)$ and $L(\Ll_1)$ are derived using the new tools from Theorem \ref{theo:formchar2}.
\begin{theo}\label{theo:charbn1}
Let $n\geq 3$, and let $\Lambda_0, \dots, \Lambda_{n}$ be the fundamental weights  and $\alpha_0, \dots, \alpha_{n}$ be the simple roots of $B_{n}^{(1)}$.
Let $\delta= \alpha_0+\alpha_1+2\alpha_2\cdots+2\alpha_n$ be the null root.
Let us set
$$q=e^{-\delta/2}, \quad c_0=1, \quad \text{and} \quad c_i = e^{\alpha_i + \cdots + \alpha_{n-1} +\alpha_n} \text{ for all } i \in \{1,\dots , n \}.$$
We have
\begin{align}
e^{-\Ll_0}\mathrm{ch}(L(\Ll_0))&= \E_{c_0,c_1, \dots , c_n}  \left((-c_0q;q^2)_{\infty} \prod_{k=1}^n (-c_k q;q^2)_{\infty} (-c_k^{-1}q;q^2)_{\infty} \right) \label{eq:b0}
\\&= \frac{1}{2} \left((-q;q^2)_{\infty} \prod_{k=1}^n (-c_k q;q^2)_{\infty} (-c_k^{-1}q;q^2)_{\infty}+ (q;q^2)_{\infty}\prod_{k=1}^n (c_k q;q^2)_{\infty} (c_k^{-1}q;q^2)_{\infty}\right), \nonumber
\\e^{-\Ll_1}\mathrm{ch}(L(\Ll_1))&= \E_{c_0,c_1, \dots , c_n}  \left((-c_0q;q^2)_{\infty} (-c_1 q^3;q^2)_{\infty} (-c_1^{-1}q^{-1};q^2)_{\infty} \prod_{k=2}^n (-c_k q;q^2)_{\infty} (-c_k^{-1}q;q^2)_{\infty} \right) \nonumber
\\&= \frac{1}{2} \left((-q;q^2)_{\infty}(-c_1 q^3;q^2)_{\infty} (-c_1^{-1}q^{-1};q^2)_{\infty} \prod_{k=2}^n (-c_k q;q^2)_{\infty} (-c_k^{-1}q;q^2)_{\infty} \right. \label{eq:b1}\\
& \left.\qquad \quad + (q;q^2)_{\infty}(c_1 q^3;q^2)_{\infty} (c_1^{-1}q^{-1};q^2)_{\infty} \prod_{k=2}^n (c_k q;q^2)_{\infty} (c_k^{-1}q;q^2)_{\infty}\right). \nonumber
\end{align}
\end{theo}

We conclude with the four level $1$ standard modules of $D_{n}^{(1)}$.
\begin{theo}\label{theo:chardn1}
Let $n\geq 4$, and let $\Lambda_0, \dots, \Lambda_{n}$ be the fundamental weights  and $\alpha_0, \dots, \alpha_{n}$ be the simple roots of $D_{n}^{(1)}$.
Let $\delta= \alpha_0+\alpha_1+2\alpha_2\cdots+2\alpha_{n-2}+\alpha_{n-1}+\alpha_n$ is the null root.
Let us set
$$q=e^{-\delta/2} \quad \text{and} \quad c_i = e^{\alpha_i + \cdots + \alpha_{n-2} +\alpha_{n-1}/2 +\alpha_n/2} \text{ for all } i \in \{1,\dots , n \}.$$
We have
\begin{align}
e^{-\Ll_0}\mathrm{ch}(L(\Ll_0))&= \E_{c_1, \dots , c_n} \left(\prod_{k=1}^n (-c_k q;q^2)_{\infty} (-c_k^{-1}q;q^2)_{\infty}\right), \label{eq:d10}
\\e^{-\Ll_1}\mathrm{ch}(L(\Ll_1))&= \E_{c_1, \dots , c_n} \left((-c_1 q^3;q^2)_{\infty} (-c_1^{-1}q^{-1};q^2)_{\infty} \prod_{k=2}^n (-c_k q;q^2)_{\infty} (-c_k^{-1}q;q^2)_{\infty}\right), \label{eq:d11}
\\e^{-\Ll_{n-1}}\mathrm{ch}(L(\Ll_{n-1})) &= \E_{c_1, \dots , c_n} \left((-c_n;q^2)_{\infty} (-c_n^{-1}q^2;q^2)_{\infty} \prod_{k=1}^{n-1} (-c_k q^2;q^2)_{\infty} (-c_k^{-1};q^2)_{\infty}\right), \label{eq:d1n-1}
\\e^{-\Ll_n}\mathrm{ch}(L(\Ll_n)) &= \E_{c_1, \dots , c_n} \left((-c_n q^2;q^2)_{\infty} (-c_n^{-1};q^2)_{\infty} \prod_{k=1}^{n-1} (-c_k q^2;q^2)_{\infty} (-c_k^{-1};q^2)_{\infty}\right). \label{eq:d1n}
\end{align}
\end{theo}
Note that these character formulas for $A_{2n}^{(2)}$, $B_n^{(1)}$, and $D_{n+1}^{(2)}$ are reminiscent of the specialised character formulas given by Bernard and Thierry-Mieg using string functions in \cite{BernardMieg}. Formulas of the same kind, but involving only products (not sums of products), can be found in Wakimoto's book \cite{WakimotoBook}.

\medskip
The paper is structured as follows. In Section \ref{sec:ground}, we recall some basics on affine Lie algebras, perfect crystals, and the theory of grounded partitions introduced in \cite{DK19-2}. In Section \ref{sec:multiground}, we introduce multi-grounded partitions and prove Theorem \ref{theo:formchar2}. In Section \ref{sec:examples}, we use our new theory to prove the character formulas of Theorems \ref{theo:chara2n2}--\ref{theo:chardn1}.

\section{Perfect crystals and grounded partitions}
\label{sec:ground}
In this section, we briefly recall the connection between grounded partitions and characters of Lie algebra modules whose ground state path is constant, introduced in our previous paper \cite{DK19-2}.
Here we only recall the major definitions. For a more detailed introduction, we refer the reader to the book of Hong and Kang \cite{HK} or to our previous paper \cite{DK19-2}. Throughout this paper, we follow the notation of \cite{HK}.

\subsection{Affine Lie algebras and character formulas}
\label{sec:lie}
We start by recalling some basic definitions on affine Lie algebras.

Let $\gf =\gf(A)$ be a Kac--Moody affine Lie algebra with generalised Cartan matrix $A = \bigl( a_{i,j}\bigr)_{i,j\in \I}$.
Let $\h $ be the Cartan subalgebra of $\gf$ and $\h^*$ be its dual. We have $ \h= \mathbb C \ot_{\Z} P^\vee$, where
$P^\vee = \Z h_0 \oplus \Z h_1 \oplus \cdots \oplus \Z h_{n-1}
\oplus \Z d$ with $h_0, \dots , h_{n-1}, d$ linearly independent. $P^\vee$ is called the \textit{dual weight lattice}.
Define linear functionals $\alpha_i$ and $\Lambda_i$ ($i \in \I$)
on $\h$ such that
\begin{equation*}
\begin{array}{cccc} \langle h_j,\alpha_i \rangle: =  \alpha_i(h_j) = a_{j,i}& \qquad &\langle d, \alpha_i \rangle: = \alpha_i(d) = \delta_{i,0}& \\
\langle h_j, \Lambda_i \rangle: = \Lambda_i(h_j) = \delta_{i,j}& \qquad &\langle d, \Lambda_i \rangle: = \Lambda_i(d) = 0 &\quad (i,j\in \I). \end{array}
\end{equation*}
The set $\Pi:= \{\alpha_i \mid i \in \I\} \subset \h^*$ is the set of simple roots,
and $\Pi^\vee := \{h_i \mid i \in \I\}\subset \h$ is the set of \textit{simple coroots}. 
We also set
$P := \{\lambda \in \mathfrak h^* \mid \lambda(P^\vee) \subset \Z\}$
to be the \textit{weight lattice}. It contains the set of {\it dominant integral weights}
$P^+ := \{\lambda \in P \mid \lambda(h_i) \in \Z_{ \geq 0} \text{ for all } i \in \I\}.$

The quintuple $(A,\Pi,\Pi^\vee,P,P^\vee)$ is called the \textit{Cartan datum} of $\gf$.

We also define the \textit{coroot lattice}
$\bar P^\vee := \Z h_0 \oplus \Z h_1 \oplus \cdots \oplus \Z h_{n-1},$
 and its complexification $\bar \h = \mathbb C \ot_{\Z}\bar P^\vee$.  The
 $\Z$-submodule
$\bar P:=\Z \Lambda_0 \oplus \Z \Lambda_1 \oplus \cdots \oplus \Z \Lambda_{n-1}$
of $P$ is called the lattice of {\it classical weights}.

Let $\bar P^+ : = \sum_{i= 0}^n  \Z_{ \geq 0} \Lambda_i$ be the set of \textit{dominant weights}.

\m
The center 
$\Z c = \{h\in P^\vee: \langle h,\alpha_i \rangle = 0 \text{ for all i } \in \I\}$
of $\gf$
is one-dimensional, generated by the {\it canonical central element}
$c=c_0 h_0 + \cdots + c_{n-1} h_{n-1}.$
The space of imaginary roots $ \Z \delta = \{\lambda\in P: \langle h_i,\lambda \rangle = 0 \text{ for all } i\in \I\}$
of $\gf$ is also one-dimensional, generated by the {\it null root} $\delta = d_0 \alpha_0 + d_1 \alpha_1 + \cdots + d_{n-1} \alpha_{n-1}.$

The {\it level} of a dominant weight $\lambda \in P^+$ is the integer $\ell$ such that
$\langle c, \lambda \rangle = \ell$.  We denote by $P_{\ell}$ (resp. $P_{\ell}^+$) the set of weights (resp. dominant weights) of level $\ell$.

\m

Let $U_q(\gf)$ (resp. $U_q'(\gf)$) be the {\it quantum affine algebra} (resp. \textit{derived quantum affine algebra}) associated to $\gf$.
Let $M$ be an integrable $U_q(\gf)$-module. It has a weight space decomposition $M = \bigoplus_{\lambda \in P} M_\lambda$,  where
$M_\lambda = \{v \in M \mid q^h\cdot v = q^{\lambda(h)}v$ for all $h \in P^\vee \}$. Assuming that $\dim M_{\lambda}<\infty$ for all $\lambda\in \mathrm{wt}(M)$, the \textit{character} of $M$ is defined by 
\begin{equation*}\label{eq:character}
 \mathrm{ch} (M) := \sum_{\lambda\in \mathrm{wt}(M)} \dim M_{\lambda} \cdot e^{\lambda},
\end{equation*}
where $\mathrm{wt}(M) = \{\lambda \in P \mid M_\lambda \neq 0\}$, and the $e^{\lambda}$'s  are formal basis elements of the group algebra $\mathbb C[\mathfrak{h}^*]$, with the multiplication defined by $e^{\lambda}e^{\mu}= e^{\lambda+\mu}$.

Let $L(\lambda)$ be an irreducible highest weight $U_q(\gf)$-module of highest weight $\lambda \in P^+$.
Then its character is given by
$$
e^{-\lambda}\mathrm{ch} (L(\lambda)) = \sum_{\mu\in \h^*} \dim L(\lambda)_{\mu} \cdot e^{\mu-\lambda} \quad \in \quad \Z[[e^{-\alpha_i}, i\in \I]].
$$
In other words, the character $e^{-\lambda}\mathrm{ch} (L(\lambda))$ is a series with positive coefficients in the $e^{-\alpha_i}$'s.
For a fixed weight $\lambda\in P$, the irreducible highest weight $\gf$-modules of weight $\lambda$ can be identified with the irreducible highest weight $U_q(\gf)$-modules of weight $\lambda$, and we have equality of characters.

\subsection{Perfect crystals}
\label{sec:crystal}
We now briefly recall notions on perfect crystals which are necessary to state the (KMN)$^2$ crystal base character formula. We assume that the reader is somewhat familiar with the basic definitions of crystal bases and quantum algebras, or can quickly catch up by reading Chapters $4$ and $10$ of \cite{HK} or our thorough introduction in \cite{DK19-2}.

Let $\mathcal O^q_{\hbox {\rm \small  int}}$ denote the category of integrable $U_q(\gf)$-modules. To each module $M = \bigoplus_{\lambda \in P} M_\lambda \in \mathcal O^q_{\hbox {\rm \small  int}}$, one can associate a corresponding crystal base $(\mathcal L, \B)$, which is unique up to isomorphism. There is a crystal graph associated to $\B$, which has vertex set $\B$, and oriented edges
$$ b \xrightarrow[]{\,\,\, i \,\,\,} b' \quad \text{if and only if}\quad \tilde f_i b = b' \text{ (or equivalently } \tilde e_i b' = b),$$
where $\tilde e_i$ and $\tilde f_i$ are the Kashiwara operators.

For $i \in \I$, the functions $\varepsilon_i, \varphi_i: \B \rightarrow \Z$ are defined as follows:
$$
\begin{array}{cc}
&\varepsilon_i(b) = \max\{ k  \geq 0 \mid \tilde e_i^k b \in \B\}, \\
&\varphi_i(b) =  \max\{ k  \geq 0 \mid \tilde f_i^k b \in \B\}. \end{array}
$$
Now define
$$\varepsilon(b) = \sum_{i=0}^{n-1} \varepsilon_i(b) \Lambda_i,
\qquad \text{and} \qquad \varphi(b) = \sum_{i=0}^{n-1} \varphi_i(b) \Lambda_i.$$
We then have $ \wt b = \varphi(b) -\varepsilon(b)$, and for all $b\in \B$ such that $\tilde e_i b \neq 0$,
$$\mathrm{wt} (\tilde e_i b) - \mathrm{wt} b = \alpha_i.$$
An energy function on $\B \ot \B$ is a map $H: \B \ot \B \rightarrow
\Z$ satisfying, for all $i \in \I$ and $b_1,b_2$ with $\tilde e(b_1 \ot b_2) \neq 0$,
$$ H\left(\tilde e_i (b_1 \ot b_2)\right)
= \begin{cases} H(b_1 \ot b_2) & \qquad \hbox{\rm if} \ \ i \neq 0, \\
H(b_1 \ot b_2) + 1 & \qquad \hbox{\rm if} \ \ i = 0 \ \hbox{\rm and} \ \varphi_0(b_1)  \geq \varepsilon_0(b_2)  \\
H(b_1 \ot b_2) - 1 & \qquad \hbox{\rm if} \ \ i = 0 \ \hbox{\rm and} \ \varphi_0(b_1)< \varepsilon_0(b_2).
\end{cases} $$
By definition, in the crystal graph of $\B \ot \B$, the value of $H(b_1\ot b_2)$ determines the values $H(b'_1\ot b'_2)$ for all vertices $b'_1\ot b'_2$ in the same connected component as $b_1\ot b_2$.
Energy functions will play a key role in the (KMN)$^2$ crystal base character formula.

\medskip
Perfect crystals, introduced by Kang, Kashiwara, Misra, Miwa, Nakashima, and Nakayashiki \cite{KMN2a,KMN2b}, provide a construction of the crystal base $\B(\lambda)$ of any irreducible $U_q(\gf)$-module
$L(\lambda)$ corresponding to a classical weight $\lambda \in \bar P^+$. 

\begin{deff}\label{deff:pc}(\cite[Definition 10.5.1]{HK})   For a positive integer $\ell$, a finite
classical crystal $\B$ is said to be a {\it perfect crystal of
level $\ell$}  for the quantum affine algebra $U_q(\gf)$ if
\begin{itemize}
\item[{\rm (1)}]  there is a finite-dimensional $U_q'(\gf)$-module
with a crystal base whose crystal graph is isomorphic to
$\B$ (when the $0$-arrows are removed);
\item[{\rm (2)}]  $\B \otimes \B$ \,  is connected;
\item[{\rm (3)}]  there exists a classical weight \,$\lambda_0$\,
such that

$$\mathrm{wt}(\B) \subset \lambda_0 + \frac{1}{d_0}\sum_{i \neq 0} \Z_{\leq 0} \alpha_i
\quad \hbox{\rm and} \quad  |\B_{\lambda_0}| = 1;$$

\item[{\rm (4)}]  for any $b \in \B$, we have
$$\langle c,\varepsilon(b)\rangle
= \sum_{i=0}^{n-1} \varepsilon_i(b) \Lambda_i(c)  \geq \ell;$$

\item[{\rm (5)}]   for each $\lambda \in \bar P_\ell^+ := \{ \mu \in \bar P^+  \mid \langle c, \mu \rangle = \ell\}$, there exist unique
vectors $b^\lambda$ and $b_\lambda$  in $\B$ such that
$\varepsilon(b^\lambda) = \lambda$ and $\varphi(b_\lambda) = \lambda$.
\end{itemize}
\end{deff}
Let us fix a perfect crystal $\B$ for the remainder of this section.
The maps $\lambda \mapsto \varepsilon(b_\lambda)$ and $\lambda\mapsto \varphi(b^\lambda)$ then define two bijections on $\bar P_\ell^+$. 
As a consequence of the vertex operator theory (\cite[(10.4.4)]{HK}), for any $\lambda \in \bar P^+_\ell$, there is a natural crystal isomorphism
\begin{gather}\label{deff:vertex} \mathcal B(\lambda) \ \iso \  \mathcal B(\varepsilon(b_{\lambda})) \ot \mathcal B \\
  \quad u_{\lambda} \, \mapsto \ \ u_{\varepsilon(b_{\lambda})}  \ot b_{\lambda}  .\nonumber 
\end{gather}

We now define the famous ground state paths and $\lambda$-paths, which are related with grounded and multi-grounded partitions.
\begin{deff}  For $\lambda \in \bar P^+_\ell$, the {\it ground state path of weight} $\lambda$ is the tensor product
 $${\p}_\lambda = \,\bigl(g_k)_{k=0}^\infty \,= \  \  \cdots \ot g_{k+1} \ot g_k \ot \cdots \ot g_1 \ot g_0,$$
 where $g_k \in \B$ for all $k \geq 0$, and
 \begin{equation}\label{eq:lamb}\begin{array} {ccc}
\lambda_0 = \lambda &\qquad g_0 = b_\lambda &  \\
\lambda_{k+1} = \varepsilon(b_{\lambda_k})
 &\qquad \, g_{k+1} = b_{\lambda_{k+1}}  & \qquad  \hbox{\rm for all}\ \
k  \geq 0\, .\end{array} \end{equation}
A tensor product $\p = (p_k)_{k=0}^\infty =  \cdots \ot p_{k+1} \ot p_k \ot \cdots \ot p_1 \ot p_0$ of elements $p_k \in \B$ is said to be a $\lambda$-{\it path} if
$p_k = g_k$ for $k$ large enough. Let $\mathcal P(\lambda)$ denote the set of $\lambda$-paths .
\end{deff}

Iterating \eqref{deff:vertex}, we obtain the following isomorphism.
 \begin{theo}
 \label{theorem:cryiso}(\cite[Theorem 10.6.4]{HK}) 
 Let $\lambda \in  \bar P^+_\ell$.
 Then there is a crystal isomorphism
 \begin{align*}
 \B(\lambda)  &\iso \mathcal P(\lambda)\\
 u_\lambda &\mapsto \p_\lambda
 \end{align*}
 between the crystal base $\B(\lambda)$ of $L(\lambda)$ and the set $\mathcal P(\lambda)$ of $\lambda$-paths.
 \end{theo}  

The crystal structure of $\mathcal P(\lambda)$ can be described as follows (\cite[(10.48)]{HK}).   
For any $\p = (p_k)_{k=0}^\infty \in \mathcal P(\lambda)$, let $N \geq 0$ be the smallest integer such that  $p_k = g_k$ for all $k  \geq N$. We have
\begin{align*}
\mathrm{wt} \p &= \lambda_N + \sum_{k=0}^{N-1} \wt p_k,\\
\tilde e_i \p\  &= \ \ \cdots \ot g_{N+1} \ot \tilde e_i\left( g_N \ot \cdots \ot p_0 \right), \\
\tilde f_i \p \ &= \ \ \cdots \ot g_{N+1} \ot \tilde f_i\left( g_N \ot \cdots \ot p_0 \right), \\
\varepsilon_i(\p) &= \max\left( \varepsilon_i(\p')-\varphi_i(g_N), 0\right),  \\
\varphi_i(\p) &= \varphi_i(\p') +  \max\left(\varphi_i(g_N)-\varepsilon_i(\p'), 0\right),
\end{align*} 
where $\p' := p_{N-1} \ot \cdots \ot p_1 \ot p_0$, and
$\wt$ is viewed as the classical weight of an element
of $\B$ or $\mathcal P(\lambda)$. \m

We are now ready to state the (KMN)$^2$ crystal base character formula, which gives an explicit expression for the affine weight  $\mathrm{wt} \p$ and connects it with the character of $L(\lambda)$.
\begin{theo}[(KMN)$^2$ crystal base character formula \cite{KMN2a}]
\label{theorem:wtchar}
Let $\lambda \in \bar P^+_{\ell}$, let $H$ be an energy function on $\B \ot \B$, and let  $\p = (p_k)_{k=0}^\infty \in \mathcal P(\lambda)$.
Then the weight of $\p$ and the character of the irreducible highest weight $U_q(\widehat \gf)$-module
$L(\lambda)$ are given by the following expressions:
\begin{align}
\mathrm{wt} \p &= \lambda + \sum_{k=0}^\infty \left(\wt p_k -\wt g_k\right)  - \frac{\delta}{d_0}\sum_{k=0}^\infty (k+1)\Big(H(p_{k+1} \ot p_k) - H(g_{k+1}\ot g_k)\Big),  \label{eq:firsteq} \\
         &= \lambda + \sum_{k=0}^\infty \left(\left(\wt p_k -\wt g_k\right)  -  \frac{\delta}{d_0}\sum_{\ell=k}^\infty (H(p_{\ell+1} \ot p_\ell) - H(g_{\ell+1}\ot g_\ell))\right), \nonumber\\
\mathrm{ch}(L(\lambda)) &= \sum_{\p \in \mathcal P(\lambda)}  e^{\mathrm{wt} \p}. \label{eq:charweight}
\end{align}
\end{theo}
A specialisation of Theorem \ref{theorem:wtchar} gives the following corollary in the special case where the ground state path is constant.

\begin{cor}\label{cor:difcond}  Suppose that $\lambda \in \bar P^+_{\ell}$ is such that $b_{\lambda}=b^{\lambda}= g$, and set $ H(g \ot g)=0$.
Then  $\ov{\rm{wt}}g = 0$, $g_k=g$ for all $k\in \mathbb{Z}_{\geq 0}$, and we have  
\begin{equation*}\label{eq:secondeq}
\mathrm{wt}\p = \lambda + \sum_{k=0}^\infty \left( \wt p_k-  \frac{\delta}{d_0} \sum_{\ell=k}^\infty H(p_{\ell+1} \ot p_\ell)\right).
\end{equation*}
\end{cor}

\subsection{Grounded partitions}
\label{sec:groundedpartitions}
To make the connection between character formulas and partitions (in particular the Primc partition identity), we introduced the concept of \textit{grounded partitions} in \cite{DK19-2}. First, recall the definition of these objects.

\begin{deff}
Let $\C$ be a set of colours, and let $\Z_{\C} = \{k_c : k \in \Z, c \in \C\}$ be the set of integers coloured with the colours of $\C$.
Let $\succ$ be a binary relation defined on $\Z_{\C}$.
A \textit{generalised coloured partition} with relation $\succ$ is a finite sequence $(\pi_0,\ldots,\pi_s)$ of coloured integers, such that for all $i \in \{0, \dots, s-1\},$ $\pi_i\succ \pi_{i+1}.$
\end{deff}

In the following, if $\pi=(\pi_0,\ldots,\pi_s)$ is a generalised coloured partition, then $c(\pi_i) \in \C$ denotes the colour of the part $\pi_i$. The quantity $|\pi|=\pi_0+\cdots+\pi_s$ is the weight of $\pi$, and $C(\pi) = c(\pi_0)\cdots c(\pi_s)$ is its colour sequence.

Fix a particular colour $\co \in \C$. Grounded partitions, which are directly related to constant ground state paths, are defined as follows.
\begin{deff}[\cite{DK19-2}] 
A \textit{grounded partition} with ground $\co$ and relation $\succ$ is a non-empty generalised coloured partition $\pi=(\pi_0,\ldots,\pi_s)$ with relation $\succ$, such that $\pi_s = 0_{\co}$, and when $s>0$, $\pi_{s-1}\neq 0_{\co}.$
\noindent Let $\Psuc$ denote the set of such partitions.
\end{deff}

For the remainder of this section, we fix $\B$ a  perfect crystal of
level $\ell$ for $U_q(\gf)$.
Let $\lambda$ be a weight of $\bar P^+_{\ell}$ such that  $b_{\lambda}=b^{\lambda}=g$, i.e. having a constant ground state path ${\p}_\lambda =\cdots \ot g \ot g \ot g.$ Let $H$ be an energy function on $\B \ot \B$ such that $H(g \ot g)=0$. Let $\Co_{\B}=\{c_b: \,b\in \B\}$ be the set of colours indexed by $\B$. 
We define the binary relation $\gtrdot$ on $\Z_{\Co_{\B}}$ by 
\begin{equation*}\label{eq:restr}
k_{c_b}\gtrdot k'_{c_{b'}} \text{ if and only if } k-k'= H(b'\ot b).
\end{equation*}

Then the set of $\lambda$-paths is in bijection with the set of grounded partitions $\Pppp$.

\begin{prop}[\cite{DK19-2}] 
\label{prop:flat1}
Let $\phi$ be the map between $\lambda$-paths and grounded partitions defined as follows:
\begin{equation*}
\phi : \quad \p\mapsto(\pi_0,\ldots,\pi_{s-1},0_{\co}),
\end{equation*}
where $\p=(p_k)_{k\geq 0}$ is a $\lambda$-path in $\Pp(\lambda)$, $s\geq 0$ is the unique non-negative integer such that $p_{s-1}\neq g$ and $p_k=g$ for all $k\geq s$, 
and for all $k\in \ssss$, the part $\pi_k$ has colour $c_{p_k}$ and size
$\sum_{l=k}^{s-1} H(p_{k+1}\ot p_k).$
Then $\phi$ is a bijection between $\Pp(\lambda)$ and $\Pppp$. Furthermore, by taking $c_{b}=e^{\wt b}$, we have for all $\pi\in \Pppp$,  
\begin{equation*}\label{eq:paca}
e^{-\lambda+\mathrm{wt} (\phi^{-1}(\pi))}  = C(\pi)e^{-\frac{\delta|\pi|}{d_0}}.
\end{equation*} 
\end{prop}

We also described a connection with the set $\Ppp$ of grounded partitions for the relation $\gg$ defined by
\begin{equation*}\label{eq:relation}
 k_{c_b}\gg k'_{c_{b'}} \text{ if and only if } k-k'\geq H(b'\ot b).
\end{equation*}
One can view the partitions of $\Pppp$ as the partitions of $\Ppp$ such that the differences between consecutive parts are minimal.
However, $\Pppp$ is not exactly the set of all minimal partitions of $\Ppp$
because, contrarily to $\Pppp$, the set $\Ppp$ has some partitions $\pi = (\pi_0,\ldots,\pi_{s-1},0_{\co})$ such that $c(\pi_{s-1})= \co$. 
Nonetheless, these two sets of grounded partitions are related by the following proposition.

\begin{prop}[\cite{DK19-2}]
\label{prop:diff}
Let $\Pp_{\co}$ be the set of grounded partitions where all parts have colour $\co$. There is a bijection $\Phi$ between $\Ppp$ and $\Pppp\times \Pp_{\co}$, such that if $\Phi(\pi) =(\mu,\nu)$, then $|\pi|=|\mu|+|\nu|$, and by setting $\co=1$, we have $C(\pi)=C(\mu)$.
\end{prop}

This allowed us to give a character formula in terms of generating functions for grounded partitions.
\begin{theo}[\cite{DK19-2}]
 \label{theo:formchar}
Let $\B$ be a perfect crystal of level $\ell$ for $U_q(\gf)$. Let $\lambda \in \overline{P}_{\ell}^+$ having a constant ground state path ${\p}_\lambda =\cdots \ot g \ot g\ot g$. Setting $q=e^{-\delta/d_0}$ and $c_b=e^{\wt b}$ for all $b\in \B$, we have $\co=1$, and the character of the irreducible highest weight $U_q(\gf)$-module
$L(\lambda)$ is given by the following expressions:
\begin{align*}
\sum_{\pi\in \Pppp} C(\pi)q^{|\pi|} &= e^{-\lambda}\mathrm{ch}(L(\lambda)),\\
 \sum_{\pi\in \Ppp} C(\pi)q^{|\pi|} &= \frac{e^{-\lambda}\mathrm{ch}(L(\lambda))}{(q;q)_{\infty}}.
\end{align*}
\end{theo}

In \cite{DK19-2}, we applied Theorem \ref{theo:formchar} to the case of the level $1$ irreducible highest weight modules of $A_{n-1}^{(1)}$, and retrieved the Kac--Peterson character formula \cite{KacPeterson}. Provided that one is able to compute the generating function for grounded partitions, it can in principle yield character formulas for all the irreducible highest weight modules of level $\ell$ having a constant ground state path.

For any Lie algebra and any level, we can always obtain a constant ground state path by considering, for any perfect crystal $\B$, the tensor product $\Bb=\B\ot \B^{\vee}$, where $\B^\vee$ is the dual of $\B$. However, it is sometimes difficult to find a nice formula for an energy function on $\Bb\ot\Bb$ given an energy function  on $\B\ot\B$. Moreover, it can also be difficult to find a nice expression for the generating function for grounded partitions corresponding to $\Bb$. Therefore Theorem \ref{theo:formchar} is difficult to apply in some cases.

We present a solution to this problem in the next section, by introducing \textit{multi-grounded partitions} which allow us to directly handle the case where the ground state path is not a constant sequence.

\section{Multi-grounded partitions}
\label{sec:multiground}
\subsection{Definition}
In the spirit of grounded partitions, we now define multi-grounded partitions. Contrarily to grounded partitions, not only their smallest part is fixed, but their $t$ smallest parts are fixed for some $t \geq 1$. This will allows us to make a connection with characters of irreducible highest weight modules having a ground state path with period $t$.

\begin{deff}\label{deff:multiground}
 Let $\Co$ be a set of colors, $\Z_{\Co}$ the set of integers coloured with colours in $\C$, and $\succ$ a binary relation defined on $\Z_{\Co}$. Suppose that there exist some colors $c_{g_0},\ldots,c_{g_{t-1}}$ in $\Co$ 
 and \textbf{unique} coloured integers $u_{c_{g_0}}^{(0)},\ldots, u_{c_{g_{t-1}}}^{(t-1)}$ such that
 \begin{align}
  &u^{(0)}+\cdots+u^{(t-1)}=0,  \label{eq:sumu} \\
  &u_{c_{g_0}}^{(0)}\succ u_{c_{g_1}}^{(1)}\succ \cdots\succ u_{c_{g_{t-1}}}^{(t-1)}\succ u_{c_{g_0}}^{(0)}.   \label{eq:cyclgrounds}
 \end{align}
Then a \textit{multi-grounded partition} with ground $c_{g_0},\ldots,c_{g_{t-1}}$ and relation $\succ$ is a non-empty generalised coloured partition $\pi = (\pi_0,\cdots,\pi_{s-1},u_{c_{g_0}}^{(0)},\ldots, u_{c_{g_{t-1}}}^{(t-1)})$ with relation $\succ$, such that
$(\pi_{s-t},\cdots,\pi_{s-1})\neq (u_{c_{g_0}}^{(0)},\ldots, u_{c_{g_{t-1}}}^{(t-1)})$ in terms of coloured integers.
\end{deff}
\noindent We denote by $\Psucm$ the set of multi-grounded partitions with ground $g_0,\ldots,g_{t-1}$ and relation $\succ$.

\begin{ex}
\label{ex}
 Let us consider the set of colours $\Co=\{c_1,c_2,c_3\}$, the matrix 
 $$
 M = 
 \begin{pmatrix}
   2&2&2\\
   0&0&2\\
   -2&0&2\\
 \end{pmatrix},$$
 and define the relation $\succ$ on $\Z_{\C}$ by $k_{c_b}\succ k'_{c_{b'}} \text{ if and only if } k-k' \geq M(b'\ot b)$.
 
If we choose $(g_0,g_1) = (1,3)$, the pair $(u^{(0)},u^{(1)})=(1,-1)$ is the unique pair satisfying \eqref{eq:sumu} and \eqref{eq:cyclgrounds}. The generalised coloured partitions
$(3_{c_3},3_{c_2},3_{c_1},-1_{c_3},1_{c_1},-1_{c_3})$  and $(1_{c_3},3_{c_1},1_{c_3},3_{c_1},-1_{c_3},1_{c_1},-1_{c_3})$
are examples of multi-grounded partitions with ground $c_1,c_3$ and relation $\succ$, while $(1_{c_1},-1_{c_3},1_{c_1},-1_{c_3})$ and $(2_{c_1},1_{c_1},-1_{c_3})$
are not.
\end{ex}

\begin{rem}
In \Def{deff:multiground}, note that Conditions \eqref{eq:sumu} and \eqref{eq:cyclgrounds} are the same for any cyclic permutation of $c_{g_0},\ldots,c_{g_{t-1}}$. Thus, if grounded partitions are well-defined for fixed grounds $c_{g_0},\ldots,c_{g_{t-1}}$ and relation $\succ$, then grounded 
partitions for any cyclic permutation 
\linebreak $c_{g_i},\ldots,c_{g_{t-1}},c_{g_0},\ldots,c_{g_{i-1}}$ of the grounds and relation $\succ$ are also well-defined. Moreover, if the unique coloured integers corresponding to the ground $c_{g_0},\ldots,c_{g_{t-1}}$ are 
$u_{c_{g_0}}^{(0)},\ldots, u_{c_{g_{t-1}}}^{(t-1)}$, then the unique coloured integers corresponding to the ground 
$c_{g_i},\ldots,c_{g_{t-1}},c_{g_0},\ldots,c_{g_{i-1}}$ 
are $u_{c_{g_i}}^{(i)},\ldots, u_{c_{g_{t-1}}}^{(t-1)}, u_{c_{g_0}}^{(0)},\ldots, u_{c_{g_{i-1}}}^{(i-1)}.$
\end{rem}

Note that grounded partitions are the particular case of multi-grounded partitions where the ground is just one colour.

\subsection{Connection with perfect crystals}
\label{section:crystalsandpartitions}
As we did with grounded partitions, we now establish a connection between multi-grounded partitions and ground state paths of perfect crystals. The difference is that now, multi-grounded partitions allow us to treat the case where the ground state path is not a constant sequence.

Let $\B$ be a perfect crystal of level $\ell$, and let $\lambda \in \bar P_{\ell}^+$ be a level $\ell$ dominant classical weight with ground state path $\p_{\lambda}=(g_k)_{k\geq 0}$.
The finitude of the set $P_{\ell}$ implies the periodicity of the sequence $(g_i)_{i\geq 0}$ (see  \eqref{eq:lamb}). Let us set $t$ to be the \textit{period} of the ground state path, i.e. the smallest non-negative integer $k$ such that $g_k=g_0$.
Let $H$ be an energy function on $\B\ot\B$. Since $\B\ot\B$ is connected, $H$ is uniquely determined by fixing its value on a particular $b_0 \ot b_0' \in \B \ot \B$. 

We now define the function $H_{\lambda}$, for all $b, b' \in \B\ot\B$, by
  \begin{equation}\label{eq:Hlamb}
   H_{\lambda} (b\ot b') := H(b\ot b')-\frac{1}{t}\sum_{k=0}^{t-1} H(g_{k+1}\ot g_k)\,.
  \end{equation}
Thus we have
\begin{equation}
\label{eq:Hlambnul}
\sum_{k=0}^{t-1} H_{\lambda}(g_{k+1}\ot g_k)=0.
\end{equation}
The function $H_{\lambda}$ satisfies all the properties of energy functions, except that it only has integer values when $t$ divides $\sum_{k=0}^{t-1} H(g_{k+1}\ot g_k).$
In the particular case where $t=1$ (where the ground state path is constant), the function $H_{\lambda}$ is the unique energy function which satisfies $H_{\lambda}(g_0\ot g_0) = 0$. This condition plays a key role in the connection between grounded partitions and $\lambda$-paths.

  Note that for any energy function $H$, we always have  
\[\sum_{k=0}^{t-1} (k+1) H_{\lambda}(g_{k+1}\ot g_k)= \sum_{k=0}^{t-1} (k+1) H(g_{k+1}\ot g_k) -\frac{t+1}{2}\sum_{k=0}^{t-1} H(g_{k+1}\ot g_k) \in \frac{1}{2}\Z  .\]
The quantity above is an integer as soon as $t$ is odd, and is equal to $0$ when $t=1$. Thus we can choose a suitable divisor $D$ of $2t$ such that 
$DH_{\lambda}(\B\ot\B)\subset \Z$ and $\frac{1}{t}\sum_{k=0}^{t-1} (k+1) DH_{\lambda}(g_{k+1}\ot g_k) \in \Z$. \textbf{In the whole paper, $D$ always denotes such an integer.}

For the particular case $t=1$, which corresponds to the case of a constant ground state path, we can always choose $D=1$.

Let us now consider the set of colours $\Co_{\B}$ indexed by $\B$, and let us define, for the remainder of this paper, the relations $\gtrdot$  and $\gg$ on $\Z_{\Co_{\B}}$ by 
\begin{align}\label{eq:DHLeg}
 k_{c_b}\gtrdot k'_{c_{b'}} &\Longleftrightarrow k-k'= DH_\lambda(b'\ot b), \\
\label{eq:DHLin}
 k_{c_b}\gg k'_{c_{b'}} &\Longleftrightarrow k-k'\geq DH_\lambda(b'\ot b).
\end{align}
We can define multi-grounded partitions associated with these relations, as can be seen in the next proposition.

\begin{prop}
\label{prop:multigroundedparts}
The set $\Pppm$ (resp. $\Ppm$) of multi-grounded partitions with ground $c_{g_0},\ldots,c_{g_{t-1}}$ and relation $\gtrdot$ (resp. $\gg$) is the set of non-empty generalised coloured partitions $\pi = (\pi_0,\cdots,\pi_{s-1},u_{c_{g_0}}^{(0)},\ldots, u_{c_{g_{t-1}}}^{(t-1)})$ with relation $\gtrdot$ (resp. $\gg$), such that
$(\pi_{s-t},\cdots,\pi_{s-1})\neq (u_{c_{g_0}}^{(0)},\ldots, u_{c_{g_{t-1}}}^{(t-1)})$, and for all $k \in \{0, \dots , t-1\},$
\begin{equation}\label{eq:choiceofu}
 u^{(k)} = -\frac{1}{t} \sum_{\ell=0}^{t-1} (\ell+1) DH_{\lambda}(g_{\ell+1}\ot g_\ell) + \sum_{\ell=k}^{t-1}DH_{\lambda}(g_{\ell+1}\ot g_\ell)\,.
\end{equation}
\end{prop}
\begin{proof}
First, we check that the colours $c_{g_0},\ldots,c_{g_{t-1}}$  and the coloured integers $u_{c_{g_0}}^{(0)},\ldots, u_{c_{g_{t-1}}}^{(t-1)}$ satisfy the conditions in \Def{deff:multiground} for both relations $\gtrdot$ and $\gg$.

We have $u^{(k)}- u^{(k+1)}=DH_{\lambda}(g_{k+1}\ot g_k)$, so \eqref{eq:cyclgrounds} is true for both $\gtrdot$ and $\gg$. To check that \eqref{eq:sumu} is true as well, we do the computation:
\begin{align*}
\sum_{k=0}^{t-1} u^{(k)} &= \sum_{k=0}^{t-1} \left(-\frac{1}{t} \sum_{\ell=0}^{t-1} (\ell+1) DH_{\lambda}(g_{\ell+1}\ot g_\ell) + \sum_{\ell=k}^{t-1}DH_{\lambda}(g_{\ell+1}\ot g_\ell)\right)\\
&=  - \sum_{\ell=0}^{t-1} (\ell+1) DH_{\lambda}(g_{\ell+1}\ot g_\ell) + \sum_{k=0}^{t-1} \sum_{\ell=k}^{t-1}DH_{\lambda}(g_{\ell+1}\ot g_\ell)\\
&=0.
\end{align*}

Moreover, the choice of integers $u^{(0)},\ldots, u^{(t-1)}$ is unique. Indeed, by 
\eqref{eq:DHLeg} and \eqref{eq:cyclgrounds}, they satisfy $u^{(k)}-u^{(k+1)} \geq DH_{\lambda}(g_{k+1}\ot g_k)$ for all $k\in \{0,\ldots,t-1\}$ (with the convention that $u^{(0)}=u^{(t)}$). Thus, by \eqref{eq:Hlambnul},
they must satisfy
 \begin{align*}
   0 &= u^{(0)}-u^{(1)}+u^{(1)}-u^{(2)}+\cdots+u^{(t-2)}-u^{(t-1)}+u^{(t-1)}-u^{(0)}\\
     &\geq DH_{\lambda}(g_{1}\ot g_0)+ DH_{\lambda}(g_{2}\ot g_1)+ \cdots+DH_{\lambda}(g_{t-1}\ot g_{t-2})+ DH_{\lambda}(g_{t}\ot g_{t-1})\\
     &=0\,,
 \end{align*}
and this implies that $u^{(k)}-u^{(k+1)} = DH_{\lambda}(g_{k+1}\ot g_k)$ for all $k\in \{0,\ldots,t-1\}$.

Finally, by \eqref{eq:sumu}, we have
\begin{align*}
 0&=u^{(0)}+\cdots+u^{(t-1)}\\
 &= u^{(0)}-u^{(1)}+ 2(u^{(1)}-u^{(2)})+\cdots+ t (u^{(t-1)}-u^{(0)}) + tu^{(0)}\,,
\end{align*}
which gives us the value of $u^{(0)}$:
$$u^{(0)} = -\frac{1}{t}\sum_{k=0}^{t-1} (k+1) DH_{\lambda}(g_{k+1}\ot g_k) .$$
The other values \eqref{eq:choiceofu} are then fully determined by the equalities $u^{(k)}-u^{(k+1)} = DH_{\lambda}(g_{k+1}\ot g_k)$.
\end{proof}

To make a more general connection between multi-grounded partitions and character formulas, we define some additional sets of multi-grounded partitions.
For any positive integer $d$, we denote by 
$^d\Ppm$ the set of multi-grounded partitions $\pi=(\pi_0,\cdots,\pi_{s-1},u_{c_{g_0}}^{(0)},\ldots, u_{c_{g_{t-1}}}^{(t-1)})$ of $\Ppm$ such that for all $k\in\{0,\ldots,s-1\}$,
\begin{equation*}
 \pi_k-\pi_{k+1} - DH_{\lambda}(p_{k+1}\ot p_{k}) \in d\Z_{\geq 0}\,,
\end{equation*}
where $c(\pi_k)=c_{p_k}$ and $\pi_s=u_{c_{g_0}}^{(0)}$. 

Finally, let $_t^d\Ppm$ (resp. $_t\Ppm$, $_t\Pppm$) denote the set of multi-grounded partitions of $^d\Ppm$ (resp. $\Ppm$, $\Pppm$) whose number of parts is divisible by $t$.

\begin{ex}
Assume that $DH_{\lambda}=M$ given in Example \ref{ex}. Then $^2\Pp_{c_1,c_3}$ is the set of multi-grounded partitions $\pi=(\pi_0,\cdots,\pi_{s-1},1_{c_1},-1_{c_3})$ of $\Pp_{c_1,c_3}$ such that for all $k\in\{0,\ldots,s-1\}$,
\begin{equation*}
 \pi_k-\pi_{k+1} - DH_{\lambda}(p_{k+1}\ot p_{k}) \in 2\Z_{\geq 0}\,,
\end{equation*}

Given that all the values in the matrix $M$ are even and that these multi-grounded partitions always end with $(1_{c_1},-1_{c_3})$, the multi-grounded partitions in $^2\Pp_{c_1,c_3}$ only have $odd$ parts.
For example, the multi-grounded partition $(7_{c_1},5_{c_2},3_{c_3},1_{c_2},1_{c_1},-1c_{c_3})$ belongs to $^2\Pp_{c_1,c_3}$. It also belongs to $_2^2\Pp_{c_1,c_3}$ as its number of parts is divisible by $2$.
\end{ex}

Now that we have introduced all the relevant notation, let us repeat the main theorem from the introduction, Theorem \ref{theo:formchar2}, which makes the connection between perfect crystals and multi-grounded partitions.

\begin{manualtheorem}{1.1}
Setting $q=e^{-\delta/(d_0D)}$ and $c_b=e^{\wt b}$ for all $b\in \B$, we have $c_{g_0}\cdots c_{g_{t-1}}=1$, and the character of the irreducible highest weight $U_q(\gf)$-module
$L(\lambda)$ is given by the following expressions:
\begin{align*}
\sum_{\mu\in _t\Pppm} C(\pi)q^{|\pi|} &= e^{-\lambda}\mathrm{ch}(L(\lambda)),\\
 \sum_{\pi\in \, _t^d\Ppm} C(\pi)q^{|\pi|} &= \frac{e^{-\lambda}\mathrm{ch}(L(\lambda))}{(q^d;q^d)_{\infty}}.
\end{align*}
\end{manualtheorem}
The remainder of this section is dedicated to the proof of Theorem \ref{theo:formchar2}.

 \bi 
Let $\p=(p_k)_{k\geq 0} \in \Pp(\lambda)$ be a $\lambda$-path  different from the ground state path $\p_{\lambda} =(g_k)_{k\geq 0}$. Then, by definition, there exists a unique positive integer $m$ such that
\begin{equation}
\label{eq:defm}
\begin{aligned}
 (p_{(m-1)t},\ldots,p_{mt-1}) &\neq (g_0,\ldots,g_{t-1})\\
 (p_{m't},\ldots,p_{(m'+1)t-1}) &= (g_0,\ldots,g_{t-1}) \qquad \text{for all } m'\geq m .
\end{aligned}
\end{equation}

We start by expression the weight of $\p$ in terms of the function $H_{\lambda}$.
\begin{lem}\label{lem:wtpm}
The weight $\mathrm{wt}(\p)$ of $\p$ is given by the following formula:
\begin{equation}\label{eq:Hfinalsum}
\mathrm{wt} (\p) = \lambda + \sum_{k=0}^{mt-1} \left( \wt(p_k) - \frac{\delta}{d_0}\left(-\frac{1}{t}\sum_{\ell=0}^{t-1} (\ell+1) H_{\lambda}(g_{\ell+1}\ot g_\ell)+\sum_{\ell=k}^{mt-1} H_{\lambda}(p_{\ell+1} \ot p_\ell) \right) \right).
\end{equation}
 \end{lem}
\begin{proof}
For any positive integer $m$, we have
  \begin{align}
  \sum_{k=0}^{mt-1} (k+1) H(g_{k+1}\ot g_k)&= \frac{m(mt+1)}{2}\sum_{k=0}^{t-1} H(g_{k+1}\ot g_k) + \sum_{k=0}^{mt-1} (k+1) H_{\lambda}(g_{k+1}\ot g_k)\quad \text{by \eqref{eq:Hlamb}}\nonumber\\
  &= \frac{m(mt+1)}{2}\sum_{k=0}^{t-1} H(g_{k+1}\ot g_k) + m\sum_{k=0}^{t-1} (k+1) H_{\lambda}(g_{k+1}\ot g_k), \label{eq:Hlambm}
  \end{align}
where the last equality follows from the periodicity of $\p_{\lambda}$.
  
  On the other hand, we have
\begin{align}
\label{eq:wtsumnul}
 \sum_{k=0}^{t-1} \wt(g_k)=0.
\end{align}
Indeed,
 \begin{align*}
  \sum_{k=0}^{t-1} \wt(g_k)&= \sum_{k=0}^{t-1} \varphi(g_k)-\varepsilon(g_k) \qquad \qquad  \text{ by defintion}\\
  &= \sum_{k=0}^{t-1} \varphi(g_k)-\varphi(g_{k+1}) \qquad \qquad \text{by \eqref{eq:lamb}}\\
  &= \varphi(g_0)-\varphi (g_{0}).&
  \end{align*}

Therefore, computing the weight $\mathrm{wt} (\p)$ given by \eqref{eq:firsteq} yields:
\begin{align*}
\mathrm{wt}(\p) &= \lambda + \sum_{k=0}^\infty \left(\wt (p_k) -\wt (g_k)\right)  - \frac{\delta}{d_0}\sum_{k=0}^\infty (k+1)\Big(H(p_{k+1} \ot p_k) - H(g_{k+1}\ot g_k)\Big)\\
 &= \lambda + \sum_{k=0}^\infty \left(\wt (p_k) -\wt (g_k)\right)  - \frac{\delta}{d_0}\sum_{k=0}^{mt-1} (k+1)\Big(H(p_{k+1} \ot p_k) - H(g_{k+1}\ot g_k)\Big)\\
 &= \lambda + \sum_{k=0}^{mt-1} \wt(p_k) - \frac{\delta}{d_0}\sum_{k=0}^{mt-1} (k+1)H_{\lambda}(p_{k+1} \ot p_k) + \frac{m\delta}{d_0} \sum_{k=0}^{t-1} (k+1) H_{\lambda}(g_{k+1}\ot g_k),
\end{align*}
where the last equality follows from \eqref{eq:Hlambm} and \eqref{eq:wtsumnul}.
Equation \eqref{eq:Hfinalsum} then follows from the fact that 
$$\sum_{k=0}^{mt-1} (k+1) H_{\lambda}(p_{k+1}\ot p_k) = \sum_{k=0}^{mt-1} \sum_{\ell=k}^{mt-1} H_{\lambda}(p_{\ell+1}\ot p_\ell).$$
\end{proof}

\bi
We now give a bijection between $\lambda$-paths and multi-grounded partitions.
\begin{prop}
\label{prop:flat2}
Let us define the map $\phi$ from $\Pp(\lambda)$ to $\Pppm$,  such that  $\phi(\p_{\lambda}) = (u_{c_{g_0}}^{(0)},\ldots, u_{c_{g_{t-1}}}^{(t-1)})$, and for all  $\p=(p_k)_{k\geq 0} \in \Pp(\Lambda)$ different from $\p_{\lambda}$ and $m$ defined in \eqref{eq:defm},
\[ \phi(\p) = (\pi_0,\cdots,\pi_{mt-1},u_{c_{g_0}}^{(0)},\ldots, u_{c_{g_{t-1}}}^{(t-1)})\]
where for all $k\in \{0,\cdots,mt-1\}$, $c(\pi_k) = c_{p_k}$ and 
\[
 \pi_k = -\frac{1}{t}\sum_{\ell=0}^{t-1} (\ell+1) DH_{\lambda}(g_{\ell+1}\ot g_\ell)+\sum_{\ell=k}^{mt-1} DH_{\lambda}(p_{\ell+1} \ot p_\ell)\,.
\]
 Then $\phi$ defines a bijection between $\Pp(\lambda)$ and the set $_t\Pppm$ of partitions of $\Pppm$ whose number of parts is divisible by $t$.
 
 Furthermore, by setting $c_b=e^{\wt(b)}$ for all $b\in \B$, we have for all $\pi \in\, _t\Pppm$, 
\begin{equation}\label{eq:pacam}
e^{-\Ll+\mathrm{wt} (\phi^{-1}(\pi))}  = C(\pi)e^{-\frac{\delta |\pi|}{d_0D}}.
\end{equation}
\end{prop}

\begin{proof}
Let $\p \in \Pp(\lambda)$ and $\pi = \phi(\p)$.
First, let us check that $\pi$ belongs to $_t\Pppm$. The multi-grounded partition $\pi$ has $(m+1)t$ parts, so its number of parts is indeed divisible by $t$.
Moreover, we have $\pi_k\gtrdot\pi_{k+1}$ for all $k\in \{0,mt-2\}$, because $\pi_k-\pi_{k+1} = DH_{\lambda}(p_{k+1} \ot p_k)$.
By \eqref{eq:Hlambnul} and \eqref{eq:choiceofu}, we have
$u^{(0)} = -\frac{1}{t}\sum_{\ell=0}^{t-1} (\ell+1) DH_{\lambda}(g_{\ell+1}\ot g_\ell)$, so that
$\pi_{mt-1} - u^{(0)} = DH_{\lambda}(p_{mt}\ot p_{mt-1})$ and $\pi_{mt-1}\gtrdot u^{(0)}_{c_{g_0}}$.
Finally, since $(p_{(m-1)t},\ldots,p_{mt-1})\neq (g_0,\ldots,g_{t-1})$, we necessarily have that
\linebreak 
$(\pi_{(m-1)t},\cdots,\pi_{mt-1})\neq (u_{c_{g_0}}^{(0)},\ldots, u_{c_{g_{t-1}}}^{(t-1)})$ in terms of coloured integers. Thus $\pi$ is indeed a multi-grounded partition belonging to $_t\Pppm$

\m Let us now give the inverse bijection $\phi^{-1}$. Start with $\pi=(\pi_0,\ldots,\pi_{mt-1},u_{c_{g_0}}^{(0)},\ldots, u_{c_{g_{t-1}}}^{(t-1)}) \in\, _t\Pppm$, with $m>0$ and colour sequence
$c_{p'_0}\cdots c_{p'_{mt-1}} c_{g_0}\cdots c_{g_{t-1}}$.
We set $\phi^{-1}(\pi) = (p_k)_{k\geq 0}$, where $p_{m't+i}= g_i$ for all $m'\geq m$ and  $i\in\{0,\ldots,t-1\}$, and $p_k= p'_k$ for all $k\in \{0,\ldots,mt-1\}$.
\begin{itemize}
 \item We first show that $(p_{(m-1)t},\ldots,p_{mt-1})\neq (g_0,\ldots,g_{t-1})$. Assume for the purpose of contradiction that $(p_{(m-1)t},\ldots,p_{mt-1})=(g_0,\ldots,g_{t-1})$.
 We then obtain by \eqref{eq:DHLeg} that 
 $$\pi_{(m-1)t+k} = -\frac{1}{t}\sum_{\ell=0}^{t-1}(\ell+1) DH_{\lambda}(g_{\ell+1}\ot g_\ell) + \sum_{\ell=k}^{t-1}DH_{\lambda}(g_{\ell+1}\ot g_\ell) = u_{c_{g_{k}}}^{(k)}.$$
 This contradicts the fact that $(\pi_{(m-1)t},\cdots,\pi_{mt-1})\neq (u_{c_{g_0}}^{(0)},\ldots, u_{c_{g_{t-1}}}^{(t-1)})$ in terms of colored integers.
 \item By \eqref{eq:DHLeg}, we also have, for all $k\in \{0,\ldots,mt-1\}$,  $\pi_k-\pi_{k+1}=DH_{\lambda}(p_{k+1}\ot p_k).$ Therefore
\begin{equation}
\label{eq:4}
\pi_k = u^{(0)} + \sum_{\ell=k}^{mt-1}(\pi_\ell-\pi_{\ell+1})= -\frac{1}{t}\sum_{\ell=0}^{t-1} (\ell+1) DH_{\lambda}(g_{\ell+1}\ot g_\ell) + \sum_{\ell=k}^{mt-1} DH_{\lambda}(p_{\ell+1}\ot p_\ell).
\end{equation}

\end{itemize}
With what precedes, we have $\phi(\phi^{-1}(\pi))=\pi$ and $\phi^{-1}(\phi(\p))=\p$.

Finally, we obtain \eqref{eq:pacam} by \Lem{lem:wtpm}. 
\end{proof}

\bi
Similarly to the case of grounded partitions, we end this section with a bijection connecting the multi-grounded partitions with relations $\gg$ and $\gtrdot$.

\begin{prop}\label{prop:diff2}
Let $^d\Pp$ be the set of classical partitions where all parts are divisible by $d$. There is a bijection $\Phi_d$ between $_t^d\Ppm$ and $_t\Pppm\times \, ^d\Pp$, such that if $\Phi_d(\pi) =(\mu,\nu)$, then $|\pi|=|\mu|+|\nu|$, and by setting $c_{g_0}\cdots c_{g_{t-1}}=1$, we have $C(\pi)=C(\mu)$.
\end{prop}
\begin{proof}
The main trick here consists in seeing classical partitions as partitions with number of parts divisible by $t$. 
If $\nu$ is a classical partition with $rt-\ell$ parts, with $r>1$ and $\ell \in \{0, \dots , t-1\}$, it suffices to add $\ell$ parts $0$ at the end of $\nu$. Then, a non-empty partition $\nu \in\ ^d\Pp$ can be uniquely written as a non-increasing sequence $\nu = (d\nu_0,\cdots,d\nu_{rt-1})$ of non-negative multiples of $d$, with $\nu_{(r-1)t}>0$.
\m 
We set $\Phi_d(u_{c_{g_0}}^{(0)},\ldots, u_{c_{g_{t-1}}}^{(t-1)})=((u_{c_{g_0}}^{(0)},\ldots, u_{c_{g_{t-1}}}^{(t-1)}),\emptyset)$. Let us now consider
$\pi=(\pi_0,\ldots,\pi_{ts-1},u_{c_{g_0}}^{(0)},\ldots, u_{c_{g_{t-1}}}^{(t-1)})$ in  $_t^d\Ppm$, 
different from $(u_{c_{g_0}}^{(0)},\ldots, u_{c_{g_{t-1}}}^{(t-1)})$, with colour sequence $c_{p'_0}\cdots c_{p'_{ts-1}}c_{g_0}\cdots c_{g_{t-1}}$, and let us build $\Phi_d(\pi) =(\mu,\nu).$ 
We set  $ \p= (p_k)_{k\geq 0}$, with $p_{s't+i}= g_{i}$ 
for all $s'\geq s$ and  $i\in \{0,\ldots,t-1\}$, and $p_k= p'_k$ for all $k\in \{0,\ldots,st-1\}$, and
$$m=\max\{k\in \{0,\ldots,s\}: (p_{(k-1)t},\ldots,p_{kt-1})\neq (g_0,\cdots,g_{t-1})\}.$$
Since $(p_{kt},\ldots,p_{(k+1)t-1})=(g_0,\cdots,g_{t-1})$ for all $k\geq m$, with the convention $c_{g_0}\cdots c_{g_{t-1}}=1$, we obtain that 
$C(\pi)=c_{p_0}\cdots c_{p_{st-1}} = c_{p_0}\cdots c_{p_{mt-1}}$.

We set 
$$\mu=(\mu_0,\ldots,\mu_{mt-1},u_{c_{g_0}}^{(0)},\ldots, u_{c_{g_{t-1}}}^{(t-1)}):= \phi(\p).$$
By \eqref{eq:4}, for all $k\in \{0,\ldots,mt-1\}$, the part $\mu_k$ has colour $c_{p_k}$ and size
$$u^{(0)}+\sum_{\ell=k}^{mt-1} DH_{\lambda}(p_{\ell+1}\ot p_\ell).$$
Thus we have $C(\pi)= C(\mu)$.
\m
Let us now build $\nu = (\nu_0,\ldots,\nu_{rt-1})$ in $^d\Pp$. We distinguish two different cases.

\begin{enumerate}
\item If $m<s$, we set $r=s$ and $\nu = (\nu_0,\ldots,\nu_{st-1})$, where 
$$\begin{cases}
\nu_k&=\ \pi_k-\mu_k \qquad \qquad \text{for } k \in \{0,\ldots,mt-1\},\\
\nu_{kt+i}&=\ \pi_k-u^{(i)}\qquad \quad \ \text{ for } k\in \{m,\ldots,s-1\}\  \text{and  } i\in \{0,\ldots,t-1\}.
\end{cases}$$
Therefore, for all $k\in \{0,\ldots,mt-2\}$, we have
\begin{align*}
\nu_{k}-\nu_{k+1} &= \pi_{k}-\pi_{k+1} - \mu_k + \mu_{k+1}&\\
&= \pi_{k}-\pi_{k+1} - DH_{\lambda}(p_{k+1}\ot p_k)&\\
&\in d \Z_{\geq 0},
\end{align*}
and 
\begin{align*}
\nu_{mt-1}-\nu_{mt} &= \pi_{mt-1}-\pi_{mt}-\mu_{mt-1}+u^{(0)}&\\
&= \pi_{k}-\pi_{k+1} - DH_{\lambda}(p_{mt}\ot p_{mt-1})\\
&\in d \Z_{\geq 0}.
\end{align*}
We also have, for all $k\in \{m,\ldots,s-1\}$ and all $i\in\{0,\ldots,t-1\}$,
\begin{align*}
\nu_{kt+i}-\nu_{kt+i+1} &= \pi_{kt+i}-\pi_{kt+i+1}-u^{(i)}+u^{(i+1)}\\
&= \pi_{kt+i}-\pi_{kt+i+1}-DH_{\lambda}(p_{kt+i+1}\ot p_{kt+i})\\
&\in d \Z_{\geq 0}.
\end{align*}
We finally observe that 
\begin{align*}
\nu_{st-1} &= \pi_{st-1}-u^{(t-1)}\\
&= \pi_{st-1}-u^{(0)}+u^{(0)}-u^{(t-1)}\\
&= \pi_{st-1}-u^{(0)}-DH_{\lambda}(p_{kt}\ot p_{kt-1})\\
&\in d \Z_{\geq 0}.
\end{align*}
We showed that the sequence $(\nu_k)_{k=0}^{st-1}$ is indeed a non-increasing sequence of multiples of $d$.

Moreover, we have $\pi_{(s-1)t}>u^{(0)}$. Indeed, $m<s$ implies that $(p'_{(s-1)t},\dots,p'_{st-1})=(g_0,\dots,g_{t-1})$. So if we had $\pi_{(s-1)t}=u^{(0)}$, then by the difference condition \eqref{eq:DHLin}, it would mean that
\linebreak
 $(\pi_{(s-1)t},\ldots,\pi_{st-1}) = (u_{c_{g_0}}^{(0)},\ldots, u_{c_{g_{t-1}}}^{(t-1)})$ as coloured integers, which contradicts the definition of multi-grounded partitions.
Thus $\nu_{(s-1)t}=\pi_{(s-1)t}-u^{(0)}>0$.

\item By definition, we have $m\leq s$, so the only other possible case is $m=s$. 
As before, we obtain that $(\pi_k-\mu_k)_{k=0}^{mt-1}$ is a non-increasing sequence of non-negative multiples of $d$. We set 
$$r=\min\{k\in \{0,\ldots,s\}: \pi_{kt}=\mu_{kt}\},$$
and $\nu_k = \pi_k-\mu_k$ for all $k\in \{0,\ldots,rt-1\}$. So, in this case too, $\nu$ belongs to $^d\Pp$ and $\nu_{(r-1)t}>0$.
\end{enumerate}

\bi
The map $\Phi_d^{-1}$ from $_t\Pppm\times \, ^d\Pp$ to $_t^d\Ppm$ simply consists in
adding the parts of\\
$\mu = (\mu_0,\ldots,\mu_{mt-1},u_{c_{g_0}}^{(0)},\ldots, u_{c_{g_{t-1}}}^{(t-1)}) \in \,_t\Pppm$ to those of $\nu = (\nu_0,\cdots,\nu_{rt-1}) \in \,^d\Pp$ to obtain a multi-grounded partition $\pi \in \,_t^d\Ppm$ in the following way: 
\begin{enumerate}
 \item if $m\geq r$, then for all $k\in \{0,\ldots,mt-1\}$, $\pi_k$ has size $\mu_k + \nu_k$ and color $c(\mu_k)$, where we set $\nu_k=0$ for all $k\in\{rt,\cdots,mt-1\}$, and we obtain the multi-grounded partition
 \[\pi = (\pi_0,\cdots,\pi_{mt-1},u_{c_{g_0}}^{(0)},\ldots, u_{c_{g_{t-1}}}^{(t-1)});\]
 \item if $m<r$, the first $mt$ parts are defined as in the case $m\geq r$, and the remaining parts are $\pi_{kt+i} = \nu_{kt+i}+u^{(i)}$ with color $c_{g_i}$ for all $k\in \{m,\ldots,r-1\}$ and $i\in\{0,\ldots,t-1\}$, and we obtain the multi-grounded partition
 \[\pi = (\pi_0,\cdots,\pi_{rt-1},u_{c_{g_0}}^{(0)},\ldots, u_{c_{g_{t-1}}}^{(t-1)}).\]
\end{enumerate}
It is easy to see that $\Phi_d$ and $\Phi_d^{-1}$ are inverses of each other, the first (resp. second) case of $\Phi_d$ being the inverse of the second (resp. first) case of $\Phi_d^{-1}$.
\end{proof}

This proposition, along with \eqref{eq:charweight} of \Thm{theorem:wtchar}, yields
Theorem \ref{theo:formchar2}.

\m Note that when $t=1$, we can choose $D=1$, and \Thm{theo:formchar} is then a particular case of \Thm{theo:formchar2}. The additional parameter $d$ allows us to have a refined equality and to simplify some calculations. It is 
particularly useful when
$DH_{\lambda}(\B\ot\B)\in d\Z$, in which case the parts of our partitions all belong to the same congruence class modulo $d$.
This is done for example in Section \ref{sec:A2n} for $A_{2n}^{(2)}$ and in Section \ref{sec:D2} for $D_{n+1}^{(2)}$.

\section{Examples of application: character formulas for several classical affine Lie algebras}
\label{sec:examples}
In this last section, we illustrate how multi-grounded partitions and \Thm{theo:formchar2} can be used to compute character formulas for standard level $1$ modules where \Thm{theo:formchar} is not easily (or not at all) applicable. All crystals can be found in the book \cite{HK}.

\subsection{The Lie algebra $A_{2n}^{(2)}(n\geq 2)$}
\label{sec:A2n}
We start by studying the algebra $A_{2n}^{(2)}$ for $n\geq 2$ to prove Theorem \ref{theo:chara2n2}.
The crystal $\B$ of the vector representation of $A_{2n}^{(2)}(n\geq 2)$ is given by the crystal graph in Figure \ref{fig1} with the weights
\begin{align*}
\wt (0)&=0,\\
\wt (u) &= -\wt (\overline{u}) = \frac{1}{2}\alpha_n+\sum_{i=u}^{n-1} \alpha_i \text{ for all $u\in \{1,\ldots,n\}$.}
\end{align*}
Here, the null root is
$\delta = \alpha_n+2\sum_{i=0}^{n-1} \alpha_i$.

\begin{figure}[H]
\begin{center}
\begin{tikzpicture}[scale=0.8, every node/.style={scale=0.8}]
\draw (-4,0) node{$\B$ :};
\draw (-4,-0.75) node{$b^{\Ll_0}=b_{\Ll_0}=0$};
\draw (-4,-1.5) node{$\p_{\Ll_0}=(\cdots000)$};
\draw (-1.5,-0.75) node {$0$};
\draw (0,0) node {$1$};
\draw (1.5,0) node {$2$};
\draw (5,0) node {$n-1$};
\draw (6.75,0) node {$n$};
\draw (3.15,0) node {$\cdots$};
\draw (0,-1.5) node {$\overline{1}$};
\draw (1.5,-1.5) node {$\overline{2}$};
\draw (5,-1.5) node {$\overline{n-1}$};
\draw (6.75,-1.5) node {$\overline{n}$};
\draw (3.15,-1.5) node {$\cdots$};

\draw (-0.2-1.5,0.2-0.75)--(0.2-1.5,0.2-0.75)--(0.2-1.5,-0.2-0.75)--(-0.2-1.5,-0.2-0.75)--(-0.2-1.5,0.2-0.75);

\foreach \x in {0,1,4.5} 
\foreach \y in {0,1}
\draw (-0.2+1.5*\x,0.2-1.5*\y)--(0.2+1.5*\x,0.2-1.5*\y)--(0.2+1.5*\x,-0.2-1.5*\y)--(-0.2+1.5*\x,-0.2-1.5*\y)--(-0.2+1.5*\x,0.2-1.5*\y);

\foreach \x in {0} 
\foreach \y  in {0,1}
\draw (-0.45+5+2*\x,0.2-1.5*\y)--(0.45+5+2*\x,0.2-1.5*\y)--(0.45+5+2*\x,-0.2-1.5*\y)--(-0.45+5+2*\x,-0.2-1.5*\y)--(-0.45+5+2*\x,0.2-1.5*\y);

\draw [thick, ->] (-1.45,-0.5)--(-0.3,-0.05);
\draw [thick, <-] (-1.45,-1)--(-0.3,-1.45);
\draw (-0.8,-0.1) node {\footnotesize{$0$}};
\draw (-0.8,-1.4) node {\footnotesize{$0$}};

\draw [thick, ->] (0.3,0)--(1.2,0);
\draw [thick, <-] (0.3,-1.5)--(1.2,-1.5);
\draw (0.75,0.15) node {\footnotesize{$1$}};
\draw (0.75,-1.67) node {\footnotesize{$1$}};

\draw [thick, ->] (1.8,0)--(2.7,0);
\draw [thick, <-] (1.8,-1.5)--(2.7,-1.5);
\draw (2.25,0.15) node {\footnotesize{$2$}};
\draw (2.25,-1.67) node {\footnotesize{$2$}};

\draw [thick, ->] (3.6,0)--(4.5,0);
\draw [thick, <-] (3.6,-1.5)--(4.5,-1.5);
\draw (4.05,0.15) node {\footnotesize{$n-2$}};
\draw (4.05,-1.67) node {\footnotesize{$n-2$}};

\draw [thick, ->] (5.6,0)--(6.4,0);
\draw [thick, <-] (5.6,-1.5)--(6.4,-1.5);
\draw (6,0.15) node {\footnotesize{$n-1$}};
\draw (6,-1.67) node {\footnotesize{$n-1$}};

\draw [thick, ->] (6.75,-0.3)--(6.75,-1.2);
\draw (6.9,-0.75) node {\footnotesize{$n$}};

\end{tikzpicture}
\end{center}
\caption{Crystal graph $\B$ of the vector representation for the Lie algebra $A_{2n}^{(2)}(n\geq 2)$}
\label{fig1}
\end{figure}
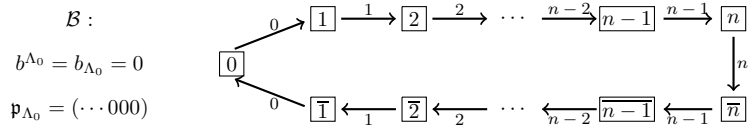

We now compute the energy function $H$ on $\B\ot \B$ such that $H(0\ot 0)=0$. To do so, we use the crystal graph of $\B\ot \B$ given in Figure \ref{fig2}. It is connected, so the energy is completely determined by the choice $H(0\ot 0)=0$. Moreover, the energy is constant on each connected component of $\B\ot \B$ after removing the $0$-arrows. More detail on the computation of energy functions can be found in \cite{HK} or \cite{DK19-2}.
 
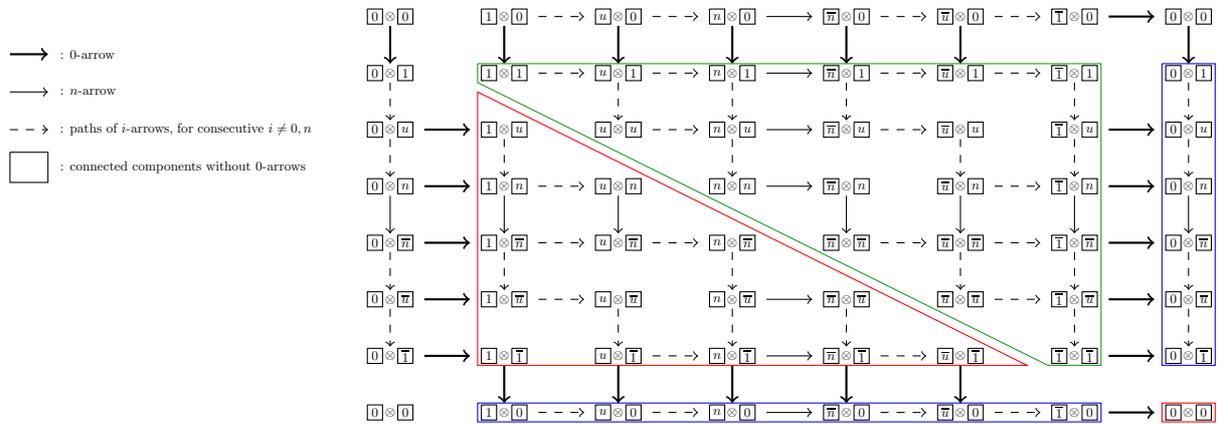
\begin{figure}[H]
\begin{center}
\begin{tikzpicture}[scale=0.5, every node/.style={scale=0.5}]

\draw [thick,->] (-10,-1)--(-9,-1);
\draw (-8.8,-1) node[right] {: $0$-arrow};

\draw [->] (-10,-2)--(-9,-2);
\draw (-8.8,-2) node[right] {: $n$-arrow};

\draw [dashed,->] (-10,-3)--(-9,-3);
\draw (-8.8,-3) node[right] {: paths of $i$-arrows, for consecutive $i \neq 0,n$};

\draw (-10,-3.6)--(-9,-3.6)--(-9,-4.4)--(-10,-4.4)--(-10,-3.6);
\draw (-8.8,-4) node[right] {: connected components without $0$-arrows};

 


\foreach \x in {0,...,7} 
\foreach \y in {0,...,7}
\draw (-0.2-0.4+3*\x,0.2-1.5*\y)--(0.2-0.4+3*\x,0.2-1.5*\y)--(0.2-0.4+3*\x,-0.2-1.5*\y)--(-0.2-0.4+3*\x,-0.2-1.5*\y)--(-0.2-0.4+3*\x,0.2-1.5*\y);

\foreach \x in {0,...,7} 
\foreach \y in {0,...,7}
\draw (-0.2+0.4+3*\x,0.2-1.5*\y)--(0.2+0.4+3*\x,0.2-1.5*\y)--(0.2+0.4+3*\x,-0.2-1.5*\y)--(-0.2+0.4+3*\x,-0.2-1.5*\y)--(-0.2+0.4+3*\x,0.2-1.5*\y);

\foreach \x in {0,...,7} 
\foreach \y in {0,...,7}
\draw (3*\x,-1.5*\y) node {$\ot$};


\foreach \x in {0,...,7} 
\draw (0.4+3*\x,0) node {$0$};
\foreach \y in {0,...,7} 
\draw (-0.4,-1.5*\y) node {$0$};

\foreach \x in {0,...,7} 
\draw (0.4+3*\x,-1.5) node {$1$};
\foreach \y in {0,...,7} 
\draw (-0.4+3,-1.5*\y) node {$1$};

\foreach \x in {0,...,7} 
\draw (0.4+3*\x,-3) node {$u$};
\foreach \y in {0,...,7} 
\draw (-0.4+6,-1.5*\y) node {$u$};

\foreach \x in {0,...,7} 
\draw (0.4+3*\x,-4.5) node {$n$};
\foreach \y in {0,...,7} 
\draw (-0.4+9,-1.5*\y) node {$n$};

\foreach \x in {0,...,7} 
\draw (0.4+3*\x,-6) node {$\overline{n}$};
\foreach \y in {0,...,7} 
\draw (-0.4+12,-1.5*\y) node {$\overline{n}$};

\foreach \x in {0,...,7} 
\draw (0.4+3*\x,-7.5) node {$\overline{u}$};
\foreach \y in {0,...,7} 
\draw (-0.4+15,-1.5*\y) node {$\overline{u}$};

\foreach \x in {0,...,7} 
\draw (0.4+3*\x,-9) node {$\overline{1}$};
\foreach \y in {0,...,7} 
\draw (-0.4+18,-1.5*\y) node {$\overline{1}$};

\foreach \x in {0,...,7} 
\draw (0.4+3*\x,-10.5) node {$0$};
\foreach \y in {0,...,7} 
\draw (-0.4+21,-1.5*\y) node {$0$};

\foreach \x in {0,...,5,7} 
\draw [thick, ->] (3*\x,-0.25)--(3*\x,-1.25);
\foreach \y in {0,2,3,4,5,6,7} 
\draw [thick, ->] (18.9,-1.5*\y)--(20.1,-1.5*\y);

\foreach \x in {0,2,3,4,6,7} 
\draw [dashed, ->] (3*\x,-1.75)--(3*\x,-2.75);
\foreach \y in {0,1,3,4,5,7} 
\draw [dashed, ->] (15.9,-1.5*\y)--(17.1,-1.5*\y);

\foreach \x in {0,1,3,5,6,7} 
\draw [dashed, ->] (3*\x,-3.25)--(3*\x,-4.25);
\foreach \y in {0,1,2,4,6,7} 
\draw [dashed, ->] (12.9,-1.5*\y)--(14.1,-1.5*\y);

\foreach \x in {0,1,2,4,5,6,7} 
\draw [->] (3*\x,-4.75)--(3*\x,-5.75);
\foreach \y in {0,1,2,3,5,6,7} 
\draw [->] (9.9,-1.5*\y)--(11.1,-1.5*\y);

\foreach \x in {0,2,3,4,6,7}
\draw [dashed, ->] (3*\x,-7.75)--(3*\x,-8.75);
\foreach \y in {0,1,2,4,6,7} 
\draw [dashed, ->] (6.9,-1.5*\y)--(8.1,-1.5*\y);

\foreach \x in {0,1,3,5,6,7} 
\draw [dashed, ->] (3*\x,-6.25)--(3*\x,-7.25);
\foreach \y in {0,1,3,4,5,7} 
\draw [dashed, ->] (3.9,-1.5*\y)--(5.1,-1.5*\y);

\foreach \x in {1,...,5} 
\draw [thick, ->] (3*\x,-9.25)--(3*\x,-10.25);
\foreach \y in {2,...,6} 
\draw [thick, ->] (0.9,-1.5*\y)--(2.1,-1.5*\y);


\foreach \x in {21}
\foreach \y in {10.5}
\draw [red] (-0.7+\x,0.25-\y)--(0.7+\x,0.25-\y)--(0.7+\x,-0.25-\y)--(-0.7+\x,-0.25-\y)--(-0.7+\x,0.25-\y);

\foreach  \x in {21}
\draw [blue](-0.7+\x,0.25-1.5)--(0.7+\x,0.25-1.5)--(0.7+\x,-0.25-9)--(-0.7+\x,-0.25-9)--(-0.7+\x,0.25-1.5);

\foreach \y in {10.5}
\draw [blue](-0.7+3,0.25-\y)--(0.7+18,0.25-\y)--(0.7+18,-0.25-\y)--(-0.7+3,-0.25-\y)--(-0.7+3,0.25-\y);

\draw [red] (-0.7+3,1-3)--(-0.7+3,-0.25-9)--(0.25+16.5,-0.25-9)--(-0.7+3,1-3);

\draw [foge] (-0.7+3,0.25-1.5)--(0.7+18,0.25-1.5)--(0.7+18,-0.25-9)--(-0.7+18,-0.25-9)--(-0.7+3,-0.25-1.5)--(-0.7+3,0.25-1.5);

\end{tikzpicture}
\end{center}
\caption{Crystal graph of $\B\ot \B$ for the Lie algebra $A_{2n}^{(2)}(n\geq 2)$}
\label{fig2}
\end{figure}

We obtain the following energy matrix:
\[
H=
\bordermatrix{
\text{}&1& 2 &\cdots &n & \overline{n} & \cdots & \overline{2} &\overline{1}&0
\cr 1&2&\cdots&\cdots&\cdots&\cdots& \cdots& \cdots&2 &1
\cr 2&0&\ddots&&&&&&\vdots &\vdots
\cr \vdots&\vdots &\ddots&\ddots&&&2^*&&\vdots&\vdots
\cr n&\vdots&&\ddots&\ddots & &&&\vdots &\vdots
\cr \overline{n}&\vdots&&&\ddots&\ddots&& &\vdots&\vdots
\cr \vdots&\vdots&&0^*&&\ddots&\ddots&&\vdots&\vdots
\cr \overline{2}&\vdots&&&& &\ddots &\ddots &\vdots &\vdots
\cr \overline{1}&0&\cdots&\cdots&\cdots&\cdots&\cdots& 0&2 &1
\cr 0&1&\cdots&\cdots&\cdots&\cdots& \cdots& \cdots&1&0
} .
\]

In this case, $L(\Lambda_0)$ is the only irreducible highest weight module of level $1$, and the corresponding ground state path is $\p_{\Ll_0}=\cdots000$, which is constant. So in theory, to obtain the character formula of \Thm{theo:chara2n2}, we could apply Theorem \ref{theo:formchar}. But here we will make use of the additional variable $d$ of Theorem \ref{theo:formchar2} to simplify our computations.

Let us apply Theorem \ref{theo:formchar2} with $D=t=1$ and $d=2$. We have $H_{\Lambda_0}=H$, and
\begin{equation}
\label{eq:th_a2}
\sum_{\pi\in \,_1^2\Pp_{c_0}^{\gg}} C(\pi)q^{|\pi|} = \frac{e^{-\Ll_0}\mathrm{ch}(L(\Lambda_0))}{(q^2;q^2)_{\infty}},
\end{equation}
where $q=e^{-\delta/2}$ and $c_b=e^{\wt b}$ for all $b\in \B$.

We recall that $_1^2\Pp_{c_0}^{\gg}$ is the set of grounded partitions $\pi=(\pi_0,\ldots,\pi_{s-1},0_{c_0})$ with relation $\gg$ and ground $c_0$, such that for all $k \in \{0, \dots , s-1\}$,
\begin{equation}
\label{eq:a2parity}
\pi_k-\pi_{k+1} -H(p_{k+1}\ot p_{k}) \in 2 \Z_{\geq0},
\end{equation}
where $c(\pi_k)=c_{p_k}$ and $\pi_s=0_{c_{g_0}}$.

These grounded partitions are exactly the partitions grounded in $c_0$ which are finite sub-sequences of 
\[\cdots\od 5_{c_1}\od 4_{c_0}\od 3_{c_{\overline{1}}}\od\cdots\od 3_{c_1}\od 2_{c_0}\od 1_{c_{\overline{1}}}\od \cdots\od 1_{c_{\overline{n}}}\od 1_{c_n}\od \cdots\od 1_{c_1}\od 0_{c_0}\,,\]
where the parts $(2k)_{c_0}$ may repeat for $k>0$.

Indeed, by \eqref{eq:a2parity}, the size difference between two consecutive parts with colours $c_b$ and $c_{b'}$
has the same parity as $H(b'\ot b)$. This implies that all the parts with colours $c_1,c_{\overline{1}},\ldots,c_n,c_{\overline{n}}$ have the same parity, different from the parity
of the parts with colour $c_0$. Since the ground has size $0$ and colour $c_0$, the parts coloured $c_0$ are even and the others are odd, and we obtain the sequence above. 
\m Setting $c_0=1$, we obtain the generating function 
\[
\sum_{\pi\in \,_1^2\Pp_{c_0}^{\gg}} C(\pi)q^{|\pi|} = \frac{(-c_1q,-c_{\overline{1}}q,\ldots,-c_nq,-c_{\overline{n}}q;q^2)_{\infty}}{(q^2;q^2)_{\infty}} .
\]
Combining this with \eqref{eq:th_a2}, we obtain
$$e^{-\Ll_0}\mathrm{ch}(L(\Lambda_0))= (-c_1q,-c_{\overline{1}}q,\ldots,-c_nq,-c_{\overline{n}}q;q^2)_{\infty},$$
which is Theorem \ref{theo:chara2n2}.

\subsection{The Lie algebra $D_{n+1}^{(2)} (n\geq 2)$}
\label{sec:D2}
We now turn to the algebra $D_{n+1}^{(2)}$ for $n\geq 2$ and prove Theorem \ref{theo:chardn2}.
The crystal $\B$ of the vector representation of $D_{n+1}^{(2)} (n\geq 2)$ is given by the crystal graph in Figure \ref{fig3} with the weights
\begin{align*}
\wt (0)&= \wt(\overline{0})= 0,\\
\wt (u) &= -\wt (\overline{u}) = \sum_{i=u}^{n} \alpha_i \text{ for all $u\in \{1,\ldots,n\}$.}
\end{align*}
Here, the null root is $\delta = \sum_{i=0}^n \alpha_i$.

\begin{figure}[H]
\begin{center}
\begin{tikzpicture}[scale=0.8, every node/.style={scale=0.8}]
\draw (-4,0) node{$\B$ :};
\draw (-4,-0.75) node{$\p_{\Ll_0}=(\cdots0\,0\,0\,0) $};
\draw (-4,-1.5) node{$\p_{\Ll_n}=(\cdots\overline{0}\,\overline{0}\,\overline{0}\,\overline{0})$};
\draw (-1.5,-0.75) node {$0$};
\draw (8.25,-0.75) node {$\overline{0}$};
\draw (0,0) node {$1$};
\draw (1.5,0) node {$2$};
\draw (5,0) node {$n-1$};
\draw (6.75,0) node {$n$};
\draw (3.15,0) node {$\cdots$};
\draw (0,-1.5) node {$\overline{1}$};
\draw (1.5,-1.5) node {$\overline{2}$};
\draw (5,-1.5) node {$\overline{n-1}$};
\draw (6.75,-1.5) node {$\overline{n}$};
\draw (3.15,-1.5) node {$\cdots$};
\draw (-0.2-1.5,0.2-0.75)--(0.2-1.5,0.2-0.75)--(0.2-1.5,-0.2-0.75)--(-0.2-1.5,-0.2-0.75)--(-0.2-1.5,0.2-0.75);
\draw (-0.2+8.25,0.2-0.75)--(0.2+8.25,0.2-0.75)--(0.2+8.25,-0.2-0.75)--(-0.2+8.25,-0.2-0.75)--(-0.2+8.25,0.2-0.75);
\foreach \x in {0,1,4.5} 
\foreach \y in {0,1}
\draw (-0.2+1.5*\x,0.2-1.5*\y)--(0.2+1.5*\x,0.2-1.5*\y)--(0.2+1.5*\x,-0.2-1.5*\y)--(-0.2+1.5*\x,-0.2-1.5*\y)--(-0.2+1.5*\x,0.2-1.5*\y);
\foreach \x in {0} 
\foreach \y  in {0,1}
\draw (-0.45+5+2*\x,0.2-1.5*\y)--(0.45+5+2*\x,0.2-1.5*\y)--(0.45+5+2*\x,-0.2-1.5*\y)--(-0.45+5+2*\x,-0.2-1.5*\y)--(-0.45+5+2*\x,0.2-1.5*\y);

\draw [thick, ->] (-1.45,-0.5)--(-0.3,-0.05);
\draw [thick, <-] (-1.45,-1)--(-0.3,-1.45);
\draw (-0.8,-0.1) node {\footnotesize{$0$}};
\draw (-0.8,-1.4) node {\footnotesize{$0$}};
\draw [thick, ->] (0.3,0)--(1.2,0);
\draw [thick, <-] (0.3,-1.5)--(1.2,-1.5);
\draw (0.75,0.15) node {\footnotesize{$1$}};
\draw (0.75,-1.67) node {\footnotesize{$1$}};
\draw [thick, ->] (1.8,0)--(2.7,0);
\draw [thick, <-] (1.8,-1.5)--(2.7,-1.5);
\draw (2.25,0.15) node {\footnotesize{$2$}};
\draw (2.25,-1.67) node {\footnotesize{$2$}};
\draw [thick, ->] (3.6,0)--(4.5,0);
\draw [thick, <-] (3.6,-1.5)--(4.5,-1.5);
\draw (4.05,0.15) node {\footnotesize{$n-2$}};
\draw (4.05,-1.67) node {\footnotesize{$n-2$}};
\draw [thick, ->] (5.6,0)--(6.4,0);
\draw [thick, <-] (5.6,-1.5)--(6.4,-1.5);
\draw (6,0.15) node {\footnotesize{$n-1$}};
\draw (6,-1.67) node {\footnotesize{$n-1$}};
\draw [thick, ->] (7.05,-0.05)--(8.2,-0.5);
\draw [thick, <-] (7.05,-1.45)--(8.2,-1);
\draw (7.6,-0.1) node {\footnotesize{$n$}};
\draw (7.6,-1.4) node {\footnotesize{$n$}};
\end{tikzpicture}
\end{center}
\caption{Crystal graph $\B$  of the vector representation for the Lie algebra $D_{n+1}^{(2)}(n\geq 2)$}
\label{fig3}
\end{figure}
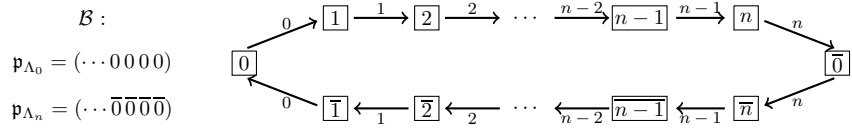

Again, we compute the energy function $H$ on $\B\ot \B$ such that $H(0\ot 0)=0$ with the help of the crystal graph of $\B\ot \B$ given in Figure \ref{fig4}.
\begin{figure}[H]
\begin{center}
\begin{tikzpicture}[scale=0.5, every node/.style={scale=0.5}]
\draw [thick,->] (-10,0)--(-9,0);
\draw (-8.8,0) node[right] {: $0$-arrow};
\draw [->] (-10,-1)--(-9,-1);
\draw (-8.8,-1) node[right] {: $n$-arrow};
\draw [->] (-10,-2)--(-9,-2);
\filldraw (-9.5,-2) circle(2pt);
\draw (-8.8,-2) node[right] {: chains of two $n$-arrows};
\filldraw (-9.5,-3) circle(2pt);
\draw (-8.8,-3) node[right] {: vertex of the form $\overline{0}\ot \cdot$ or $\cdot\ot\overline{0}$};
\draw [dashed,->] (-10,-4)--(-9,-4);
\draw (-8.8,-4) node[right] {: paths of $i$-arrows, for consecutive $i\neq 0,n$};
\draw (-10,-4.6)--(-9,-4.6)--(-9,-5.4)--(-10,-5.4)--(-10,-4.6);
\draw (-8.8,-5) node[right] {: connected components without $0$-arrows};
\foreach \x in {0,...,7} 
\foreach \y in {0,...,7}
\draw (-0.2-0.4+3*\x,0.2-1.5*\y)--(0.2-0.4+3*\x,0.2-1.5*\y)--(0.2-0.4+3*\x,-0.2-1.5*\y)--(-0.2-0.4+3*\x,-0.2-1.5*\y)--(-0.2-0.4+3*\x,0.2-1.5*\y);
\foreach \x in {0,...,7} 
\foreach \y in {0,...,7}
\draw (-0.2+0.4+3*\x,0.2-1.5*\y)--(0.2+0.4+3*\x,0.2-1.5*\y)--(0.2+0.4+3*\x,-0.2-1.5*\y)--(-0.2+0.4+3*\x,-0.2-1.5*\y)--(-0.2+0.4+3*\x,0.2-1.5*\y);

\foreach \x in {0,...,7} 
\foreach \y in {0,...,7}
\draw (3*\x,-1.5*\y) node {$\ot$};

\foreach \x in {0,...,7} 
\draw (0.4+3*\x,0) node {$0$};
\foreach \y in {0,...,7} 
\draw (-0.4,-1.5*\y) node {$0$};

\foreach \x in {0,...,7} 
\draw (0.4+3*\x,-1.5) node {$1$};
\foreach \y in {0,...,7} 
\draw (-0.4+3,-1.5*\y) node {$1$};

\foreach \x in {0,...,7} 
\draw (0.4+3*\x,-3) node {$u$};
\foreach \y in {0,...,7} 
\draw (-0.4+6,-1.5*\y) node {$u$};

\foreach \x in {0,...,7} 
\draw (0.4+3*\x,-4.5) node {$n$};
\foreach \y in {0,...,7} 
\draw (-0.4+9,-1.5*\y) node {$n$};

\foreach \x in {0,...,7} 
\draw (0.4+3*\x,-6) node {$\overline{n}$};
\foreach \y in {0,...,7} 
\draw (-0.4+12,-1.5*\y) node {$\overline{n}$};

\foreach \x in {0,...,7} 
\draw (0.4+3*\x,-7.5) node {$\overline{u}$};
\foreach \y in {0,...,7} 
\draw (-0.4+15,-1.5*\y) node {$\overline{u}$};

\foreach \x in {0,...,7} 
\draw (0.4+3*\x,-9) node {$\overline{1}$};
\foreach \y in {0,...,7} 
\draw (-0.4+18,-1.5*\y) node {$\overline{1}$};

\foreach \x in {0,...,7} 
\draw (0.4+3*\x,-10.5) node {$0$};
\foreach \y in {0,...,7} 
\draw (-0.4+21,-1.5*\y) node {$0$};

\foreach \x in {0,...,5,7} 
\draw [thick, ->] (3*\x,-0.25)--(3*\x,-1.25);
\draw [thick, ->] (3*3.5,-0.25)--(3*3.5,-1.25);
\foreach \y in {0,2,3,4,5,6,7} 
\draw [thick, ->] (18.9,-1.5*\y)--(20.1,-1.5*\y);
\draw [thick, ->] (18.9,-1.5*3.5)--(20.1,-1.5*3.5);

\foreach \x in {0,2,3,4,6,7} 
\draw [dashed, ->] (3*\x,-1.75)--(3*\x,-2.75);
\draw [dashed, ->] (3*3.5,-1.75)--(3*3.5,-2.75);
\foreach \y in {0,1,3,4,5,7} 
\draw [dashed, ->] (15.9,-1.5*\y)--(17.1,-1.5*\y);
\draw [dashed, ->] (15.9,-1.5*3.5)--(17.1,-1.5*3.5);

\foreach \x in {0,1,3,5,6,7} 
\draw [dashed, ->] (3*\x,-3.25)--(3*\x,-4.25);
\draw [dashed, ->] (3*3.5,-3.25)--(3*3.5,-4.25);
\foreach \y in {0,1,2,4,6,7} 
\draw [dashed, ->] (12.9,-1.5*\y)--(14.1,-1.5*\y);
\draw [dashed, ->] (12.9,-1.5*3.5)--(14.1,-1.5*3.5);

\foreach \x in {0,1,2,4,5,6,7} 
\draw [->] (3*\x,-4.75)--(3*\x,-5.75);
\foreach \y in {0,1,2,3,5,6,7} 
\draw [->] (9.9,-1.5*\y)--(11.1,-1.5*\y);

\foreach \x in {0,2,3,4,6,7}
\draw [dashed, ->] (3*\x,-7.75)--(3*\x,-8.75);
\draw [dashed, ->] (3*3.5,-7.75)--(3*3.5,-8.75);
\foreach \y in {0,1,2,4,6,7} 
\draw [dashed, ->] (6.9,-1.5*\y)--(8.1,-1.5*\y);
\draw [dashed, ->] (6.9,-1.5*3.5)--(8.1,-1.5*3.5);

\foreach \x in {0,1,3,5,6,7} 
\draw [dashed, ->] (3*\x,-6.25)--(3*\x,-7.25);
\draw [dashed, ->] (3*3.5,-6.25)--(3*3.5,-7.25);
\foreach \y in {0,1,3,4,5,7} 
\draw [dashed, ->] (3.9,-1.5*\y)--(5.1,-1.5*\y);
\draw [dashed, ->] (3.9,-1.5*3.5)--(5.1,-1.5*3.5);

\foreach \x in {1,...,5} 
\draw [thick, ->] (3*\x,-9.25)--(3*\x,-10.25);
\draw [thick, ->] (3*3.5,-9.25)--(3*3.5,-10.25);
\foreach \y in {2,...,6} 
\draw [thick, ->] (0.9,-1.5*\y)--(2.1,-1.5*\y);
\draw [thick, ->] (0.9,-1.5*3.5)--(2.1,-1.5*3.5);

\draw [->] (9.2,-5.25)--(10.3,-5.25); 
\draw [->] (10.5,-5.35)--(10.5,-5.9);

\foreach \x in {21}
\foreach \y in {10.5}
\draw [red] (-0.7+\x,0.25-\y)--(0.7+\x,0.25-\y)--(0.7+\x,-0.25-\y)--(-0.7+\x,-0.25-\y)--(-0.7+\x,0.25-\y);

\foreach  \x in {21}
\draw [blue](-0.7+\x,0.25-1.5)--(0.7+\x,0.25-1.5)--(0.7+\x,-0.25-9)--(-0.7+\x,-0.25-9)--(-0.7+\x,0.25-1.5);

\foreach \y in {10.5}
\draw [blue](-0.7+3,0.25-\y)--(0.7+18,0.25-\y)--(0.7+18,-0.25-\y)--(-0.7+3,-0.25-\y)--(-0.7+3,0.25-\y);

\draw [red] (-0.7+3,1-3)--(-0.7+3,-0.25-9)--(0.25+16.5,-0.25-9)--(0.25+10.5,-0.25-6)--(0.25+10.5,0.25-5.25)--(-0.7+9,0.25-5.25)--(-0.7+3,1-3);

\draw [foge] (-0.7+3,0.25-1.5)--(0.7+18,0.25-1.5)--(0.7+18,-0.25-9)--(-0.7+18,-0.25-9)--(-0.7+12,-0.25-6)--(-0.7+12,-0.25-4.5)--(-0.7+9,-0.25-4.5)--(-0.7+3,-0.25-1.5)--(-0.7+3,0.25-1.5);

\foreach \x in {0,...,7} 
\filldraw (3*\x,-5.25) circle(2pt);
\foreach \y in {0,1,2,3,3.5,4,5,6,7} 
\filldraw (10.5,-1.5*\y) circle(2pt);

\end{tikzpicture}
\end{center}
\caption{Crystal graph of $\B\ot \B$ for the Lie algebra $D_{n+1}^{(2)} (n\geq 2)$}
\label{fig4}
\end{figure}
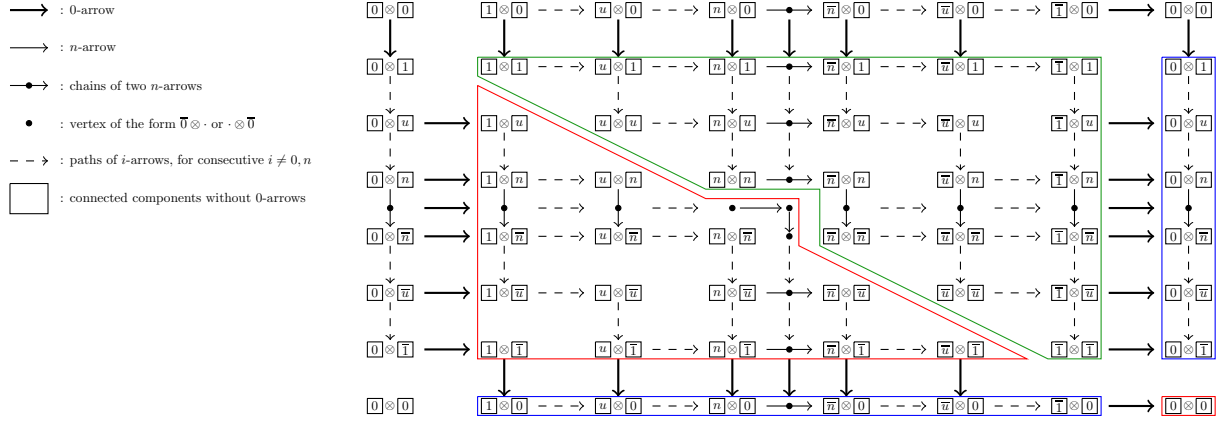

We obtain the following energy matrix:
\begin{equation}\label{eq:energyD}
H= \bordermatrix{
\text{}&1&2&\cdots&n&\overline{0}&\overline{n}&\cdots&\overline{2}&\overline{1}&0
\cr 1 &2&\cdots &\cdots&2&2&2&\cdots & \cdots &2&1
\cr 2& 0 &\ddots&2^*&\vdots&\vdots&\vdots&&&\vdots&\vdots
\cr \vdots&\vdots&\ddots&\ddots &\vdots&\vdots&\vdots&2^*&&\vdots&\vdots
\cr n&\vdots&&\ddots&2&2&2&\cdots&\cdots&2&1
\cr \overline{0}&\vdots&&&\ddots&0&2&\cdots&\cdots&2&1
\cr \overline{n}&\vdots&&&&\ddots&2&\cdots&\cdots&2&\vdots
\cr \vdots&\vdots&&0^*&&&\ddots&\ddots&2^*&\vdots&\vdots
\cr \overline{2}&\vdots&&&&& &\ddots &\ddots &\vdots &\vdots
\cr \overline{1}&0&\cdots&\cdots&\cdots&\cdots&\cdots&\cdots& 0&2 &1
\cr 0&1&\cdots&\cdots&\cdots&\cdots& \cdots & \cdots& \cdots&1&0
} .
\end{equation}

In this case, there are two irreducible highest weight modules of level $1$: $L(\Lambda_0)$ and $L(\Lambda_n)$, and the corresponding ground state paths are $\p_{\Ll_0}=\cdots000$ and $\p_{\Ll_{n}}=\cdots\overline{0}\overline{0}\overline{0}$, respectively. Again, though both ground state paths are constant, we make use of the additional variable $d$ of Theorem \ref{theo:formchar2} to simplify our computations.

\subsubsection{Character for $L(\Lambda_0)$}
We start by studying $L(\Lambda_0)$.
Again, we apply Theorem \ref{theo:formchar2} with $D=t=1$ and $d=2$. We have $H_{\Lambda_0}=H$, and
\begin{equation}
\label{eq:th_d2}
\sum_{\pi\in \,_1^2\Pp_{c_0}^{\gg}} C(\pi)q^{|\pi|} = \frac{e^{-\Ll_0}\mathrm{ch}(L(\Lambda_0))}{(q^2;q^2)_{\infty}},
\end{equation}
where $q=e^{-\delta}$ and $c_b=e^{\wt b}$ for all $b\in \B$.

Note that the energy matrix \eqref{eq:energyD} is exactly the same as in the case of 
$A_{2n}^{(2)}$ except that the row and column $\overline{0}$ were added. 
So, if we remove the parts coloured $c_{\overline{0}}$, the grounded partitions of $_1^2\Pp_{c_0}^{\gg}$ are the same as in the previous section. Moreover, the added parts coloured $c_{\overline{0}}$ can repeat arbitrarily many times, and are located between $c_n$ and $c_{\overline{n}}$ in the order of colours. Thus, $_1^2\Pp_{c_0}^{\gg}$ is the set of grounded partitions with ground $c_0$, which are finite subsequences of 
\[\cdots\od 3_{c_{\overline{1}}}\od \cdots\od 3_{c_{\overline{n}}}\od 3_{c_{\overline{0}}}\od 3_{c_n}\od \cdots\od 3_{c_1}\od 2_{c_0}\od 1_{c_{\overline{1}}}\od \cdots\od 1_{c_{\overline{n}}}\od 1_{c_{\overline{0}}}\od 1_{c_n}\od \cdots\od 1_{c_1}\od 0_{c_0}\cdot\]
where only the parts $2k_{c_0}$ and $(2k-1)_{c_{\overline{0}}}$ may repeat for all $k>0$.

\m Thus we obtain  the generating function 
\[
\sum_{\pi\in \,_1^2\Pp_{c_0}^{\gg}} C(\pi)q^{|\pi|} = \frac{(-c_1q,-c_{\overline{1}}q,\ldots,-c_nq,-c_{\overline{n}}q;q^2)_{\infty}}{(c_0q^2;q^2)_{\infty}(c_{\overline{0}}q;q^2)_{\infty}},
\]
and setting $c_0 = c_{\overline{0}}=1$, we get
\[
\sum_{\pi\in \,_1^2\Pp_{c_0}^{\gg}} C(\pi)q^{|\pi|} = \frac{(-c_1q,-c_{\overline{1}}q,\ldots,-c_nq,-c_{\overline{n}}q;q^2)_{\infty}}{(q;q)_{\infty}},
\]
Combining this with \eqref{eq:th_d2}, we obtain
$$e^{-\Ll_0}\mathrm{ch}(L(\Lambda_0))= \frac{(-c_1q,-c_{\overline{1}}q,\ldots,-c_nq,-c_{\overline{n}}q;q^2)_{\infty}}{(q;q^2)_{\infty}},$$
which is \eqref{eq:d21}.

\subsubsection{Character for $L(\Lambda_n)$}
We now turn to $L(\Lambda_n)$.
Again, we apply Theorem \ref{theo:formchar2} with $D=t=1$ and $d=2$. We have $H_{\Lambda_n}=H$, and
\begin{equation}
\label{eq:th_d22}
\sum_{\pi\in \,_1^2\Pp_{c_{\overline{0}}}^{\gg}} C(\pi)q^{|\pi|} = \frac{e^{-\Ll_n}\mathrm{ch}(L(\Lambda_n))}{(q^2;q^2)_{\infty}},
\end{equation}
where $q=e^{-\delta}$ and $c_b=e^{\wt b}$ for all $b\in \B$.

The energy matrix, and therefore difference conditions, are still the same. But now the grounded partitions are grounded in $c_{\overline{0}}$ instead of $c_0$. The grounded partitions have smallest part $0_{c_{\overline{0}}}$, so the parts with colours $c_{\overline{0}},c_1,c_{\overline{1}},\ldots,c_n,c_{\overline{n}}$ are now even, while the parts with colour $c_0$ are odd. 
Thus the grounded partitions in $\,_1^2\Pp_{c_{\overline{0}}}^{\gg}$ are exactly the partitions grounded in $c_{\overline{0}}$, which are finite subsequences of 
\[\cdots\od 4_{c_1}\od 3_{c_0}\od 2_{c_{\overline{1}}}\od \cdots\od 2_{c_{\overline{n}}}\od 2_{c_{\overline{0}}}\od 2_{c_n}\od \cdots\od 2_{c_1}\od 1_{c_0}\od 0_{c_{\overline{1}}}\od \cdots\od 0_{c_{\overline{n}}}\od 0_{c_{\overline{0}}}\cdot\]
where for all $k>0$, the parts $(2k-1)_{c_0}$ and $2k_{c_{\overline{0}}}$ may repeat.

\m Thus we obtain that the generating function 
\[
\sum_{\pi\in \,_1^2\Pp_{c_{\overline{0}}}^{\gg}} C(\pi)q^{|\pi|} = \frac{(-c_1q^2,-c_{\overline{1}},\ldots,-c_nq^2,-c_{\overline{n}};q^2)_{\infty}}{(c_0q;q^2)_{\infty}(c_{\overline{0}}q^2;q^2)_{\infty}}\,,
\]
and setting $c_0 = c_{\overline{0}}=1$, we get
\[
\sum_{\pi\in \,_1^2\Pp_{c_{\overline{0}}}^{\gg}} C(\pi)q^{|\pi|} = \frac{(-c_1q^2,-c_{\overline{1}},\ldots,-c_nq^2,-c_{\overline{n}};q^2)_{\infty}}{(q;q)_{\infty}}\,,
\]
Combining this with \eqref{eq:th_d22}, we obtain
$$e^{-\Ll_n}\mathrm{ch}(L(\Lambda_n))=  \frac{(-c_1q^2,-c_{\overline{1}},\ldots,-c_nq^2,-c_{\overline{n}};q^2)_{\infty}}{(q;q^2)_{\infty}},$$
which is \eqref{eq:d22}.

\subsection{The Lie algebra $A_{2n-1}^{(2)}(n\geq 3)$}
\label{Sec:A2n-1}
We now move to our first example where the ground state paths are not constant: the Lie algebra $A_{2n-1}^{(2)}$ for $n\geq 3$, and prove Theorem \ref{theo:chara2n1}.

The crystal $\B$ of the vector representation of $A_{2n-1}^{(2)}(n\geq 3)$ is given by the crystal graph in Figure \ref{fig5} with, for all $u\in \{1,\ldots,n\}$, the weights
$$\wt (u) = -\wt (\overline{u}) = \frac{1}{2}\alpha_n+\sum_{i=u}^{n-1} \alpha_i .$$
Here, the null root is
\[\delta = \alpha_0+\alpha_1+\alpha_n+2\sum_{i=2}^{n-1} \alpha_i .\]

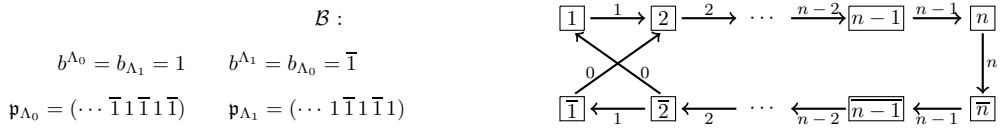
\begin{figure}[H]
\begin{center}
\begin{tikzpicture}[scale=0.8, every node/.style={scale=0.8}]
\draw (-4,0) node{$\B$ :};
\draw (-6,-0.75) node{$b^{\Ll_0}=b_{\Ll_1}=1\qquad b^{\Ll_1}=b_{\Ll_0}=\overline{1}$};
\draw (-6,-1.5) node{$\p_{\Ll_0}=(\cdots\,\overline{1}\,1\,\overline{1}\,1\,\overline{1})\qquad \p_{\Ll_1}=(\cdots\,1\,\overline{1}\,1\,\overline{1}\,1)$};
\draw (0,0) node {$1$};
\draw (1.5,0) node {$2$};
\draw (5,0) node {$n-1$};
\draw (6.75,0) node {$n$};
\draw (3.15,0) node {$\cdots$};
\draw (0,-1.5) node {$\overline{1}$};
\draw (1.5,-1.5) node {$\overline{2}$};
\draw (5,-1.5) node {$\overline{n-1}$};
\draw (6.75,-1.5) node {$\overline{n}$};
\draw (3.15,-1.5) node {$\cdots$};

\foreach \x in {0,1,4.5} 
\foreach \y in {0,1}
\draw (-0.2+1.5*\x,0.2-1.5*\y)--(0.2+1.5*\x,0.2-1.5*\y)--(0.2+1.5*\x,-0.2-1.5*\y)--(-0.2+1.5*\x,-0.2-1.5*\y)--(-0.2+1.5*\x,0.2-1.5*\y);
\foreach \x in {0} 
\foreach \y  in {0,1}
\draw (-0.45+5+2*\x,0.2-1.5*\y)--(0.45+5+2*\x,0.2-1.5*\y)--(0.45+5+2*\x,-0.2-1.5*\y)--(-0.45+5+2*\x,-0.2-1.5*\y)--(-0.45+5+2*\x,0.2-1.5*\y);

\draw [thick, <-] (0.05,-0.25)--(1.45,-1.25);
\draw [thick, <-] (1.45,-0.25)--(0.05,-1.25);
\draw (0.3,-0.9) node {\footnotesize{$0$}};
\draw (1.2,-0.9) node {\footnotesize{$0$}};

\draw [thick, ->] (0.3,0)--(1.2,0);
\draw [thick, <-] (0.3,-1.5)--(1.2,-1.5);
\draw (0.75,0.15) node {\footnotesize{$1$}};
\draw (0.75,-1.67) node {\footnotesize{$1$}};
\draw [thick, ->] (1.8,0)--(2.7,0);
\draw [thick, <-] (1.8,-1.5)--(2.7,-1.5);
\draw (2.25,0.15) node {\footnotesize{$2$}};
\draw (2.25,-1.67) node {\footnotesize{$2$}};
\draw [thick, ->] (3.6,0)--(4.5,0);
\draw [thick, <-] (3.6,-1.5)--(4.5,-1.5);
\draw (4.05,0.15) node {\footnotesize{$n-2$}};
\draw (4.05,-1.67) node {\footnotesize{$n-2$}};
\draw [thick, ->] (5.6,0)--(6.4,0);
\draw [thick, <-] (5.6,-1.5)--(6.4,-1.5);
\draw (6,0.15) node {\footnotesize{$n-1$}};
\draw (6,-1.67) node {\footnotesize{$n-1$}};

\draw [thick, ->] (6.75,-0.25)--(6.75,-1.25);
\draw (6.9,-0.75) node {\footnotesize{$n$}};

\end{tikzpicture}
\end{center}
\caption{Crystal graph $\B$  of the vector representation for the Lie algebra $A_{2n-1}^{(2)}(n\geq 3)$}
\label{fig5}
\end{figure}

We use the crystal graph for $\B\ot \B$ given in Figure \ref{fig6} to compute the energy function $H$ on $\B\ot \B$ such that  $H(1\ot \overline{1})=-1$.

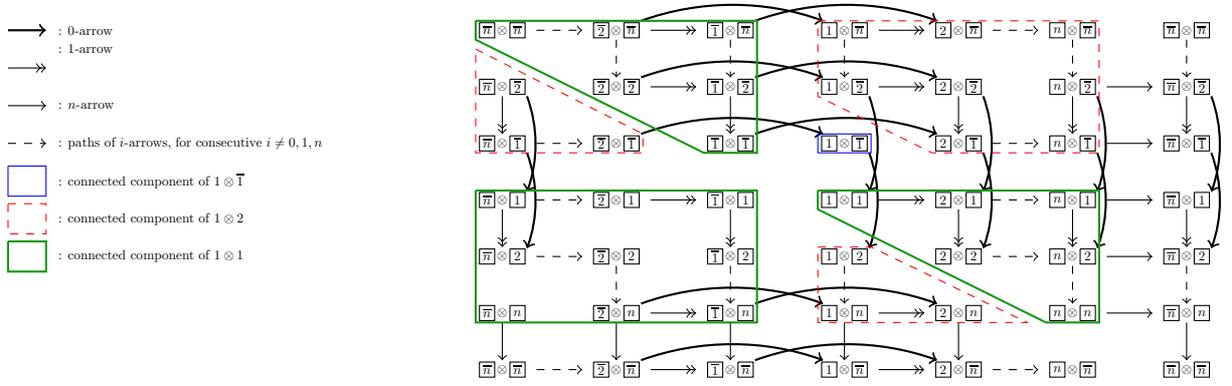
\begin{figure}[H]
\begin{center}
\begin{tikzpicture}[scale=0.5, every node/.style={scale=0.5}]
\draw [thick,->] (-10,-1-0.5)--(-9,-1-0.5);
\draw (-8.8,-1-0.5) node[right] {: $0$-arrow};
\draw [->>] (-10,-2-0.5)--(-9,-2-0.5);
\draw (-8.8,-2) node[right] {: $1$-arrow};
\draw [->] (-10,-3-0.5)--(-9,-3-0.5);
\draw (-8.8,-3-0.5) node[right] {: $n$-arrow};
\draw [dashed,->] (-10,-4-0.5)--(-9,-4-0.5);
\draw (-8.8,-4-0.5) node[right] {: paths of $i$-arrows, for consecutive $i\neq 0,1,n$};
\draw [blue] (-10,-4.6-0.5)--(-9,-4.6-0.5)--(-9,-5.4-0.5)--(-10,-5.4-0.5)--(-10,-4.6-0.5);
\draw (-8.8,-5-0.5) node[right] {: connected component of $1\ot \overline{1}$};
\draw [dashed,red] (-10,-5.6-0.5)--(-9,-5.6-0.5)--(-9,-6.4-0.5)--(-10,-6.4-0.5)--(-10,-5.6-0.5);
\draw (-8.8,-6-0.5) node[right] {: connected component of $1\ot 2$};
\draw [thick,foge] (-10,-6.6-0.5)--(-9,-6.6-0.5)--(-9,-7.4-0.5)--(-10,-7.4-0.5)--(-10,-6.6-0.5);
\draw (-8.8,-7-0.5) node[right] {: connected component of $1\ot 1$};
\foreach \x in {1,...,7} 
\foreach \y in {1,...,7}
\draw (-0.2-0.4+3*\x,0.2-1.5*\y)--(0.2-0.4+3*\x,0.2-1.5*\y)--(0.2-0.4+3*\x,-0.2-1.5*\y)--(-0.2-0.4+3*\x,-0.2-1.5*\y)--(-0.2-0.4+3*\x,0.2-1.5*\y);
\foreach \x in {1,...,7} 
\foreach \y in {1,...,7}
\draw (-0.2+0.4+3*\x,0.2-1.5*\y)--(0.2+0.4+3*\x,0.2-1.5*\y)--(0.2+0.4+3*\x,-0.2-1.5*\y)--(-0.2+0.4+3*\x,-0.2-1.5*\y)--(-0.2+0.4+3*\x,0.2-1.5*\y);

\foreach \x in {1,...,7} 
\foreach \y in {1,...,7}
\draw (3*\x,-1.5*\y) node {$\ot$};

\foreach \x in {1,...,7} 
\draw (0.4+3*\x,-10.5) node {$\overline{n}$};
\foreach \y in {1,...,7} 
\draw (-0.4+21,-1.5*\y) node {$\overline{n}$};

\foreach \x in {1,...,7}
\draw (0.4+3*\x,-1.5) node {$\overline{n}$};
\foreach \y in {1,...,7} 
\draw (-0.4+3,-1.5*\y) node {$\overline{n}$};

\foreach \x in {1,...,7}
\draw (0.4+3*\x,-3) node {$\overline{2}$};
\foreach \y in {1,...,7} 
\draw (-0.4+6,-1.5*\y) node {$\overline{2}$};

\foreach \x in {1,...,7}
\draw (0.4+3*\x,-4.5) node {$\overline{1}$};
\foreach \y in {1,...,7} 
\draw (-0.4+9,-1.5*\y) node {$\overline{1}$};

\foreach \x in {1,...,7}
\draw (0.4+3*\x,-6) node {$1$};
\foreach \y in {1,...,7} 
\draw (-0.4+12,-1.5*\y) node {$1$};

\foreach \x in {1,...,7}
\draw (0.4+3*\x,-7.5) node {$2$};
\foreach \y in {1,...,7} 
\draw (-0.4+15,-1.5*\y) node {$2$};

\foreach \x in {1,...,7}
\draw (0.4+3*\x,-9) node {$n$};
\foreach \y in {1,...,7} 
\draw (-0.4+18,-1.5*\y) node {$n$};

\foreach \x in {1,...,5,7} 
\draw [->] (3*\x,-9.25)--(3*\x,-10.25);

\foreach \x in {2,3,4,6,7} 
\draw [dashed,->] (3*\x,-1.75)--(3*\x,-2.75);
\foreach \x in {2,3,4,6,7} 
\draw [dashed,->] (3*\x,-7.75)--(3*\x,-8.75);

\foreach \x in {2,4}
\foreach \y in {1,2,4,6,7}
\draw [->>] (0.9+3*\x,-1.5*\y)--(2.1+3*\x,-1.5*\y);

\foreach \x in {2,3}
\foreach \y in {1,2,3,6,7}
\draw [thick,->] (0.65+3*\x,0.25-1.5*\y) arc (110:70:7);

\foreach \y in {2,4}
\foreach \x in {1,3,5,6,7}
\draw [->>] (3*\x,-0.25-1.5*\y)--(3*\x,-1.25-1.5*\y);

\foreach \y in {2,3}
\foreach \x in {1,4,5,6,7}
\draw [thick,->] (0.65+3*\x,-0.25-1.5*\y) arc (20:-20:3.7);

\foreach \y in {2,...,6} 
\draw [->] (18.9,-1.5*\y)--(20.1,-1.5*\y);

\foreach \y in {1,3,4,5,7} 
\draw [dashed,->] (3.9,-1.5*\y)--(5.1,-1.5*\y);
\foreach \y in {1,3,4,5,7} 
\draw [dashed,->] (15.9,-1.5*\y)--(17.1,-1.5*\y);

\draw [blue] (-0.7+12,0.25-4.5)--(0.7+12,0.25-4.5)--(0.7+12,-0.25-4.5)--(-0.7+12,-0.25-4.5)--(-0.7+12,0.25-4.5);

\draw [dashed,red] (-0.7+12,0.25-1.5)--(0.7+18,0.25-1.5)--(0.7+18,-0.25-4.5)--(-0.7+15,-0.25-4.5)--(-0.7+12,-0.25-3)--(-0.7+12,0.25-1.5);
\draw [dashed,red] (-0.7+3,1-3)--(0.7+6,0.25-4.5)--(0.7+6,-0.25-4.5)--(-0.7+3,-0.25-4.5)--(-0.7+3,1-3);
\draw [dashed,red] (-0.7+12,0.25-7.5)--(0.7+12,0.25-7.5)--(-1.2+18,-0.25-9)--(-0.7+12,-0.25-9)--(-0.7+12,0.25-7.5);

\draw [thick,foge] (-0.7+3,0.25-6)--(0.7+9,0.25-6)--(0.7+9,-0.25-9)--(-0.7+3,-0.25-9)--(-0.7+3,0.25-6);
\draw [thick,foge] (-0.7+3,0.25-1.5)--(0.7+9,0.25-1.5)--(0.7+9,-0.25-4.5)--(-0.7+9,-0.25-4.5)--(-0.7+3,-0.25-1.5)--(-0.7+3,0.25-1.5);
\draw [thick,foge] (-0.7+12,0.25-6)--(0.7+18,0.25-6)--(0.7+18,-0.25-9)--(-0.7+18,-0.25-9)--(-0.7+12,-0.25-6)--(-0.7+12,0.25-6);
\end{tikzpicture}
\end{center}
\caption{Crystal graph of $\B\ot \B$ for the Lie algebra $A_{2n-1}^{(2)} (n\geq 3)$}
\label{fig6}
\end{figure}

The energy function such that $H(1\ot \overline{1})=-1$ is given by the following matrix:
\begin{equation}
\label{eq:Ha2}
H=
\bordermatrix{
\text{}&1& 2 &\cdots &n & \overline{n} & \cdots & \overline{2} &\overline{1}
\cr 1&1&\cdots&\cdots&\cdots&\cdots& \cdots& \cdots&1 
\cr 2&0&\ddots&&&&&&\vdots 
\cr \vdots&\vdots &\ddots&\ddots&&&1^*&&\vdots
\cr n&\vdots&&\ddots&\ddots & &&&\vdots
\cr \overline{n}&\vdots&&&\ddots&\ddots&& &\vdots
\cr \vdots&\vdots&&0^*&&\ddots&\ddots&&\vdots
\cr \overline{2}&0&0&&& &\ddots &\ddots &\vdots
\cr \overline{1}&-1&0&\cdots&\cdots&\cdots&\cdots& 0&1
} .
\end{equation}

These difference conditions correspond the partial order 
\[\cdots \ll \begin{array}{c} 0_{c_{\overline{1}}} \\ 1_{c_1}\end{array}\ll 1_{c_2}\ll\cdots \ll 1_{c_n}\ll 1_{c_{\overline{n}}}\ll \cdots \ll 1_{c_{\overline{2}}}\ll\begin{array}{c} 1_{c_{\overline{1}}} \\ 2_{c_1}\end{array}\ll 2_{c_2}\ll\cdots .\]

In this case, there are two irreducible highest weight modules of level $1$: $L(\Lambda_0)$ and $L(\Lambda_1)$, with corresponding ground state paths $\p_{\Ll_0}=\cdots 1 \overline{1}1 \overline{1}1 \overline{1}$ and $\p_{\Ll_{1}}=\cdots\overline{1}1\overline{1}1\overline{1}1$.

\subsubsection{Character for $L(\Lambda_0)$}
\label{sec:A0}
We start with $L(\Lambda_0)$.
Here we refer to the notation of \Sct{sec:multiground}. 
Recall that the ground state path of $\Lambda_0$ is $\p_{\Lambda_0}=(g_k)_{k=0}^\infty$ with $g_{2k}=\overline{1}$ and $g_{2k+1}=1$ for all $k\geq 0$. Here, the period of the ground state path is $t=2$, and our choice of particular value $H(1\ot \overline{1})=-1$ for the energy function gives
\begin{equation*}\label{eq:1v1}
H(g_{2k+2}\ot g_{2k+1})=-H(g_{2k+1}\ot g_{2k})=1 .
\end{equation*}
Thus we have $H(1\ot \overline{1})+H(\overline{1} \ot 1)=0$, and by \eqref{eq:Hlamb},  $H_{\Lambda_0}=H$. 
By \eqref{eq:choiceofu}, we obtain that $u^{(0)}=-1$ and $u^{(1)}=1$.

We apply \Thm{theo:formchar2} with $d=2$ and $D=2$, which is allowed because $H(g_{1}\ot g_0)+2H(g_2\ot g_1)=-1$. We obtain
\begin{equation}
\label{eq:tha210}
\sum_{\pi\in \, _2^2\Pp_{c_{\overline{1}}c_1}^{\gg}} C(\pi)q^{|\pi|} = \frac{e^{-\Ll_0}\mathrm{ch}(L(\Lambda_0))}{(q^2;q^2)_{\infty}},
\end{equation}
where $q=e^{-\delta/2}$ and $c_b=e^{\wt b}$ for all $b\in \B$.

\m Recall that $_2^2\Pp_{c_{\overline{1}}c_1}^{\gg}$ is the set of multi-grounded partitions $\pi=(\pi_0,\ldots,\pi_{2s-1},-1_{c_{\overline{1}}},1_{c_1})$ with relation $\gg$ and grounds $c_{\overline{1}},c_1$, \textbf{having an even number of parts}, such that for all $k \in \{0, \dots , 2s-1\}$,
\begin{equation}
\label{eq:a21parity}
\pi_k-\pi_{k+1} - 2H(p_{k+1}\ot p_{k}) \in 2 \Z_{\geq0},
\end{equation}
where $c(\pi_k)=c_{p_k}$ and $\pi_{2s}=-1_{c_{\overline{1}}}$.

We observe that, by \eqref{eq:a21parity} and the fact that $u^{(0)}=-1$, the multi-grounded partitions of $_2^2\Pp_{c_{\overline{1}}c_1}^{\gg}$ have parts with odd sizes, as the differences between consecutive parts are even and the grounds' sizes are odd (indeed, we always have the fixed tail $((-1)_{c_{\overline{1}}},1_{c_1})$).
Besides, computing the generating function for partitions in $_2^2\Pp_{c_{\overline{1}}c_1}^{\gg}$ is not too difficult. It suffices to notice that, combined with \eqref{eq:a21parity}, $\gg$ is the following partial order the set of coloured odd integers:
\[\begin{array}{c} (-1)_{c_{\overline{1}}} \\ 1_{c_1}\end{array}\ll 1_{c_2}\ll\cdots \ll 1_{c_n}\ll 1_{c_{\overline{n}}}\ll \cdots \ll 1_{c_{\overline{2}}}\ll\begin{array}{c} 1_{c_{\overline{1}}} \\ 3_{c_1}\end{array}\ll 3_{c_2}\ll\cdots .\]
 We also note that, since $H(b\ot b)=1$ for all $b\in \B$, only parts coloured $c_1$ and $c_{\overline{1}}$ can appear several times, in sequences of the form
 \begin{equation*}\label{eq:twistedseq}
 \cdots \ll (2k-1)_{c_{\overline{1}}}\ll (2k+1)_{c_1}\ll(2k-1)_{c_{\overline{1}}}\ll \cdots \ll(2k-1)_{c_{\overline{1}}}\ll (2k+1)_{c_1}\ll\cdots \cdot 
 \end{equation*}
The generating function of these sequences for a fixed integer $k \geq 1$ is given by
$$ \frac{(1+c_{\overline{1}}q^{2k-1})(1+c_1q^{2k+1})}{(1-c_{\overline{1}}c_1q^{4k})}\, ,$$
where the denominator generates pairs $((2k-1)_{c_{\overline{1}}},(2k+1)_{c_1})$ that can repeat arbitrarily many times, and the numerator accounts for the possibility of having an isolated $(2k+1)_{c_1}$ on the left end of the sequence, or an isolated $(2k-1)_{c_{\overline{1}}}$ on the right end of the sequence.


Note that for $k=0$, only the sequence $(1_{c_1},(-1)_{c_{\overline{1}}},1_{c_1})$ can occur at the tail of the partitions grounded in $c_{\overline{1}},c_1$, but not the sequence $((-1)_{c_{\overline{1}}},1_{c_1},(-1)_{c_{\overline{1}}},1_{c_1})$, as this would violate the definition of multi-grounded partitions.
So, if we temporarily forget the condition on the even number of parts in $_2^2\Pp_{c_{\overline{1}}c_1}^{\gg}$, the generation function would be 
\[(1+c_1q)\cdot\frac{(-c_1q^3,-c_{\overline{1}}q,-c_2q,-c_{\overline{2}}q,\ldots,-c_nq,-c_{\overline{n}}q;q^2)_{\infty}}{(c_{\overline{1}}c_1q^4;q^4)_{\infty}} =\frac{(-c_1q,-c_{\overline{1}}q,\ldots,-c_nq,-c_{\overline{n}}q;q^2)_{\infty}}{(c_{\overline{1}}c_1q^4;q^4)_{\infty}} .\]
Now to take into account the fact that there are an even number of parts, we use \eqref{eq:evennumberparts}.
Thus the multi-grounded partitions in $_2^2\Pp_{c_{\overline{1}}c_1}^{\gg}$ have the following generating function:
\begin{equation*}
 \sum_{\pi \in \,_2^2\Pp_{c_{\overline{1}}c_1}^{\gg}} C(\pi)q^{|\pi|}= \frac{(-c_1q,-c_{\overline{1}}q,\ldots,-c_nq,-c_{\overline{n}}q;q^2)_{\infty}+(c_1q,c_{\overline{1}}q,\ldots,c_nq,c_{\overline{n}}q;q^2)_{\infty}}{2(c_{\overline{1}}c_1q^4;q^4)_{\infty}} .
\end{equation*}

The final expression \eqref{eq:a210} follows by using \eqref{eq:tha210}.

\subsubsection{Character for $L(\Lambda_1)$}
We now turn to $L(\Lambda_1)$, with a similar reasoning. 
Recall that the ground state path of $\Lambda_1$ is $(g_k)_{k=0}^\infty$ with $g_{2k+1}=\overline{1}$ and $g_{2k}=1$ for all $k\geq 0$. We still have $H_{\Lambda_1}=H$, and by setting $D=2$, we have
by \eqref{eq:choiceofu} that $u^{0}=1$ and $u^{(1)}=-1$.
\Thm{theo:formchar2} gives
\begin{equation*}
\label{eq:tha211}
\sum_{\pi\in \, _2^2\Pp_{c_1c_{\overline{1}}}^{\gg}} C(\pi)q^{|\pi|} = \frac{e^{-\Ll_1}\mathrm{ch}(L(\Lambda_1))}{(q^2;q^2)_{\infty}},
\end{equation*}
where $q=e^{-\delta/2}$ and $c_b=e^{\wt b}$ for all $b\in \B$.

So we need to study the set of multi-grounded partitions with ground $c_1,c_{\overline{1}}$ corresponding to $_2^2\Pp_{c_1c_{\overline{1}}}^{\gg}$.
We have almost the same set of partitions as in $_2^2\Pp_{c_{\overline{1}}c_1}^{\gg}$, except that now the tail is always $(1_{c_1},(-1)_{\overline{1}})$, and we can end with the sequence $((-1)_{\overline{1}},1_{c_1},(-1)_{\overline{1}})$, but 
not with $(1_{c_1},(-1)_{\overline{1}},1_{c_1},(-1)_{\overline{1}})$.

Thus the generating function where we temporarily omit the condition on the parity of the number of parts is given by 
\[(1+c_{\overline{1}}q^{-1})\cdot\frac{(-c_1q^3,-c_{\overline{1}}q,-c_2q,-c_{\overline{2}}q,\ldots,-c_nq,-c_{\overline{n}}q;q^2)_{\infty}}{(c_{\overline{1}}c_1q^4;q^4)_{\infty}} =\frac{(-c_1q^3,-c_{\overline{1}}q^{-1},-c_2q,-c_{\overline{2}}q,\ldots,-c_nq,-c_{\overline{n}}q;q^2)_{\infty}}{(c_{\overline{1}}c_1q^4;q^4)_{\infty}}.\]
So the multi-grounded partitions in $_2^2\Pp_{c_1c_{\overline{1}}}^{\gg}$ (with the condition on the even number of parts) are generated by
\begin{equation*}
\sum_{\pi\in \, _2^2\Pp_{c_1c_{\overline{1}}}^{\gg}} C(\pi)q^{|\pi|}= \frac{(-c_1q^3,-c_{\overline{1}}q^{-1},-c_2q,-c_{\overline{2}}q,\ldots,-c_nq,-c_{\overline{n}}q;q^2)_{\infty}+(c_1q^3,c_{\overline{1}}q^{-1},c_2q,c_{\overline{2}}q,\ldots,c_nq,c_{\overline{n}}q;q^2)_{\infty}}{2(c_{\overline{1}}c_1q^4;q^4)_{\infty}}.
\end{equation*}

\subsection{The Lie algebra $B_{n}^{(1)}(n\geq 3)$}
We now study the Lie algebra $B_{n}^{(1)}$ for $n\geq 3$, which has standard level $1$ modules with constant and with non-constant ground state paths.

The crystal $\B$ of the vector representation of $B_{n}^{(1)}(n\geq 3)$ is given by the crystal graph in Figure \ref{fig7} with $\wt (0)=0$ and  for all $u\in \{1,\ldots,n\}$
$$\wt (u) = -\wt (\overline{u}) = \sum_{i=u}^{n} \alpha_i .$$
Here, the null root is $\delta = \alpha_0+\alpha_1+2\sum_{i=2}^{n} \alpha_i$.

\begin{figure}[H]
\begin{center}
\begin{tikzpicture}[scale=0.8, every node/.style={scale=0.8}]
\draw (-4,0) node{$\B$ :};
\draw (-4,-0.75) node{$\p_{\Ll_n}=(\cdots\,0\,0\,0\,0)$};
\draw (-4,-1.5) node{$\p_{\Ll_1}=(\cdots\,1\,\overline{1}\,1\,\overline{1}\,1)$};
\draw (-4,-2.25) node{$\p_{\Ll_0}=(\cdots\,\overline{1}\,1\,\overline{1}\,1\,\overline{1})$};
\draw (8.25,-0.75) node {$0$};
\draw (0,0) node {$1$};
\draw (1.5,0) node {$2$};
\draw (5,0) node {$n-1$};
\draw (6.75,0) node {$n$};
\draw (3.15,0) node {$\cdots$};
\draw (0,-1.5) node {$\overline{1}$};
\draw (1.5,-1.5) node {$\overline{2}$};
\draw (5,-1.5) node {$\overline{n-1}$};
\draw (6.75,-1.5) node {$\overline{n}$};
\draw (3.15,-1.5) node {$\cdots$};

\draw (-0.2+8.25,0.2-0.75)--(0.2+8.25,0.2-0.75)--(0.2+8.25,-0.2-0.75)--(-0.2+8.25,-0.2-0.75)--(-0.2+8.25,0.2-0.75);
\foreach \x in {0,1,4.5} 
\foreach \y in {0,1}
\draw (-0.2+1.5*\x,0.2-1.5*\y)--(0.2+1.5*\x,0.2-1.5*\y)--(0.2+1.5*\x,-0.2-1.5*\y)--(-0.2+1.5*\x,-0.2-1.5*\y)--(-0.2+1.5*\x,0.2-1.5*\y);
\foreach \x in {0} 
\foreach \y  in {0,1}
\draw (-0.45+5+2*\x,0.2-1.5*\y)--(0.45+5+2*\x,0.2-1.5*\y)--(0.45+5+2*\x,-0.2-1.5*\y)--(-0.45+5+2*\x,-0.2-1.5*\y)--(-0.45+5+2*\x,0.2-1.5*\y);

\draw [thick, <-] (0.05,-0.25)--(1.45,-1.25);
\draw [thick, <-] (1.45,-0.25)--(0.05,-1.25);
\draw (0.3,-0.9) node {\footnotesize{$0$}};
\draw (1.2,-0.9) node {\footnotesize{$0$}};

\draw [thick, ->] (0.3,0)--(1.2,0);
\draw [thick, <-] (0.3,-1.5)--(1.2,-1.5);
\draw (0.75,0.15) node {\footnotesize{$1$}};
\draw (0.75,-1.67) node {\footnotesize{$1$}};
\draw [thick, ->] (1.8,0)--(2.7,0);
\draw [thick, <-] (1.8,-1.5)--(2.7,-1.5);
\draw (2.25,0.15) node {\footnotesize{$2$}};
\draw (2.25,-1.67) node {\footnotesize{$2$}};
\draw [thick, ->] (3.6,0)--(4.5,0);
\draw [thick, <-] (3.6,-1.5)--(4.5,-1.5);
\draw (4.05,0.15) node {\footnotesize{$n-2$}};
\draw (4.05,-1.67) node {\footnotesize{$n-2$}};
\draw [thick, ->] (5.6,0)--(6.4,0);
\draw [thick, <-] (5.6,-1.5)--(6.4,-1.5);
\draw (6,0.15) node {\footnotesize{$n-1$}};
\draw (6,-1.67) node {\footnotesize{$n-1$}};
\draw [thick, ->] (7.05,-0.05)--(8.2,-0.5);
\draw [thick, <-] (7.05,-1.45)--(8.2,-1);
\draw (7.6,-0.1) node {\footnotesize{$n$}};
\draw (7.6,-1.4) node {\footnotesize{$n$}};

\end{tikzpicture}
\end{center}
\caption{Crystal graph $\B$  of the vector representation for the Lie algebra $B_{n}^{(1)}(n\geq 3)$}
\label{fig7}
\end{figure}
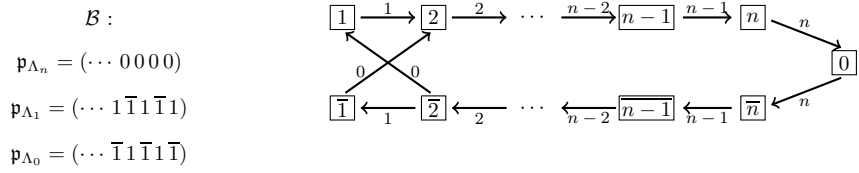

We compute the energy function $H$ on $\B\ot \B$ such that $H(0\ot 0)=0$ with the help of the crystal graph for $\B\ot \B$ given in Figure \ref{fig8}.
\begin{figure}[H]
\begin{center}
\begin{tikzpicture}[scale=0.5, every node/.style={scale=0.5}]
\draw [thick,->] (-10,-1)--(-9,-1);
\draw (-8.8,-1) node[right] {: $0$-arrow};
\draw [->>] (-10,-2)--(-9,-2);
\draw (-8.8,-2) node[right] {: $1$-arrow};
\draw [->] (-10,-3)--(-9,-3);
\draw (-8.8,-3) node[right] {: $n$-arrow};
\draw [dashed,->] (-10,-4)--(-9,-4);
\draw (-8.8,-4) node[right] {: paths of $i$-arrows, for consecutive $i\neq 0,1,n$};
\draw [blue] (-10,-4.6)--(-9,-4.6)--(-9,-5.4)--(-10,-5.4)--(-10,-4.6);
\draw (-8.8,-5) node[right] {: connected component of $1\ot \overline{1}$};
\draw [dashed,red] (-10,-5.6)--(-9,-5.6)--(-9,-6.4)--(-10,-6.4)--(-10,-5.6);
\draw (-8.8,-6) node[right] {: connected component of $1\ot 2$};
\draw [thick,foge] (-10,-6.6)--(-9,-6.6)--(-9,-7.4)--(-10,-7.4)--(-10,-6.6);
\draw (-8.8,-7) node[right] {: connected component of $1\ot 1$};
\foreach \x in {0,...,7} 
\foreach \y in {0,...,7}
\draw (-0.2-0.4+3*\x,0.2-1.5*\y)--(0.2-0.4+3*\x,0.2-1.5*\y)--(0.2-0.4+3*\x,-0.2-1.5*\y)--(-0.2-0.4+3*\x,-0.2-1.5*\y)--(-0.2-0.4+3*\x,0.2-1.5*\y);
\foreach \x in {0,...,7} 
\foreach \y in {0,...,7}
\draw (-0.2+0.4+3*\x,0.2-1.5*\y)--(0.2+0.4+3*\x,0.2-1.5*\y)--(0.2+0.4+3*\x,-0.2-1.5*\y)--(-0.2+0.4+3*\x,-0.2-1.5*\y)--(-0.2+0.4+3*\x,0.2-1.5*\y);

\foreach \x in {0,...,7} 
\foreach \y in {0,...,7}
\draw (3*\x,-1.5*\y) node {$\ot$};

\foreach \x in {0,...,7} 
\draw (0.4+3*\x,0) node {$0$};
\foreach \y in {0,...,7} 
\draw (-0.4,-1.5*\y) node {$0$};
\foreach \x in {0,...,7} 
\draw (0.4+3*\x,-10.5) node {$0$};
\foreach \y in {0,...,7} 
\draw (-0.4+21,-1.5*\y) node {$0$};

\foreach \x in {0,...,7}
\draw (0.4+3*\x,-1.5) node {$\overline{n}$};
\foreach \y in {0,...,7} 
\draw (-0.4+3,-1.5*\y) node {$\overline{n}$};

\foreach \x in {0,...,7}
\draw (0.4+3*\x,-3) node {$\overline{2}$};
\foreach \y in {0,...,7} 
\draw (-0.4+6,-1.5*\y) node {$\overline{2}$};

\foreach \x in {0,...,7}
\draw (0.4+3*\x,-4.5) node {$\overline{1}$};
\foreach \y in {0,...,7} 
\draw (-0.4+9,-1.5*\y) node {$\overline{1}$};

\foreach \x in {0,...,7}
\draw (0.4+3*\x,-6) node {$1$};
\foreach \y in {0,...,7} 
\draw (-0.4+12,-1.5*\y) node {$1$};

\foreach \x in {0,...,7}
\draw (0.4+3*\x,-7.5) node {$2$};
\foreach \y in {0,...,7} 
\draw (-0.4+15,-1.5*\y) node {$2$};

\foreach \x in {0,...,7}
\draw (0.4+3*\x,-9) node {$n$};
\foreach \y in {0,...,7} 
\draw (-0.4+18,-1.5*\y) node {$n$};

\foreach \x in {0,...,5,7} 
\draw [->] (3*\x,-0.25)--(3*\x,-1.25);
\foreach \x in {1,...,5} 
\draw [->] (3*\x,-9.25)--(3*\x,-10.25);

\foreach \x in {0,2,3,4,6,7} 
\draw [dashed,->] (3*\x,-1.75)--(3*\x,-2.75);
\foreach \x in {0,2,3,4,6,7} 
\draw [dashed,->] (3*\x,-7.75)--(3*\x,-8.75);

\foreach \x in {2,4}
\foreach \y in {0,1,2,4,6,7}
\draw [->>] (0.9+3*\x,-1.5*\y)--(2.1+3*\x,-1.5*\y);

\foreach \x in {2,3}
\foreach \y in {0,1,2,3,6,7}
\draw [thick,->] (0.65+3*\x,0.25-1.5*\y) arc (110:70:7);

\foreach \y in {2,4}
\foreach \x in {0,1,3,5,6,7}
\draw [->>] (3*\x,-0.25-1.5*\y)--(3*\x,-1.25-1.5*\y);

\foreach \y in {2,3}
\foreach \x in {0,1,4,5,6,7}
\draw [thick,->] (0.65+3*\x,-0.25-1.5*\y) arc (20:-20:3.7);

\foreach \y in {2,...,6} 
\draw [->] (0.9,-1.5*\y)--(2.1,-1.5*\y);
\foreach \y in {0,2,3,4,5,6,7} 
\draw [->] (18.9,-1.5*\y)--(20.1,-1.5*\y);

\foreach \y in {0,1,3,4,5,7} 
\draw [dashed,->] (3.9,-1.5*\y)--(5.1,-1.5*\y);
\foreach \y in {0,1,3,4,5,7} 
\draw [dashed,->] (15.9,-1.5*\y)--(17.1,-1.5*\y);

\draw [blue] (-0.7+12,0.25-4.5)--(0.7+12,0.25-4.5)--(0.7+12,-0.25-4.5)--(-0.7+12,-0.25-4.5)--(-0.7+12,0.25-4.5);

\draw [dashed,red] (-0.7+12,0.25-1.5)--(0.7+21,0.25-1.5)--(0.7+21,-0.25-4.5)--(-0.7+15,-0.25-4.5)--(-0.7+12,-0.25-3)--(-0.7+12,0.25-1.5);
\draw [dashed,red] (-0.7+3,1-3)--(0.7+6,0.25-4.5)--(0.7+6,-0.25-4.5)--(-0.7+3,-0.25-4.5)--(-0.7+3,1-3);
\draw [dashed,red] (-0.7+12,0.25-7.5)--(0.7+12,0.25-7.5)--(0.7+18,0.25-10.5)--(0.7+21,0.25-10.5)--(0.7+21,-0.25-10.5)--(-0.7+12,-0.25-10.5)--(-0.7+12,0.25-7.5);

\draw [thick,foge] (-0.7+3,0.25-6)--(0.7+9,0.25-6)--(0.7+9,-0.25-10.5)--(-0.7+3,-0.25-10.5)--(-0.7+3,0.25-6);
\draw [thick,foge] (-0.7+3,0.25-1.5)--(0.7+9,0.25-1.5)--(0.7+9,-0.25-4.5)--(-0.7+9,-0.25-4.5)--(-0.7+3,-0.25-1.5)--(-0.7+3,0.25-1.5);
\draw [thick,foge] (-0.7+12,0.25-6)--(0.7+21,0.25-6)--(0.7+21,-0.25-9)--(-0.7+18,-0.25-9)--(-0.7+12,-0.25-6)--(-0.7+12,0.25-6);
\end{tikzpicture}
\end{center}
\caption{Crystal graph of $\B\ot \B$ for the Lie algebra $B_{n}^{(1)} (n\geq 3)$}
\label{fig8}
\end{figure}
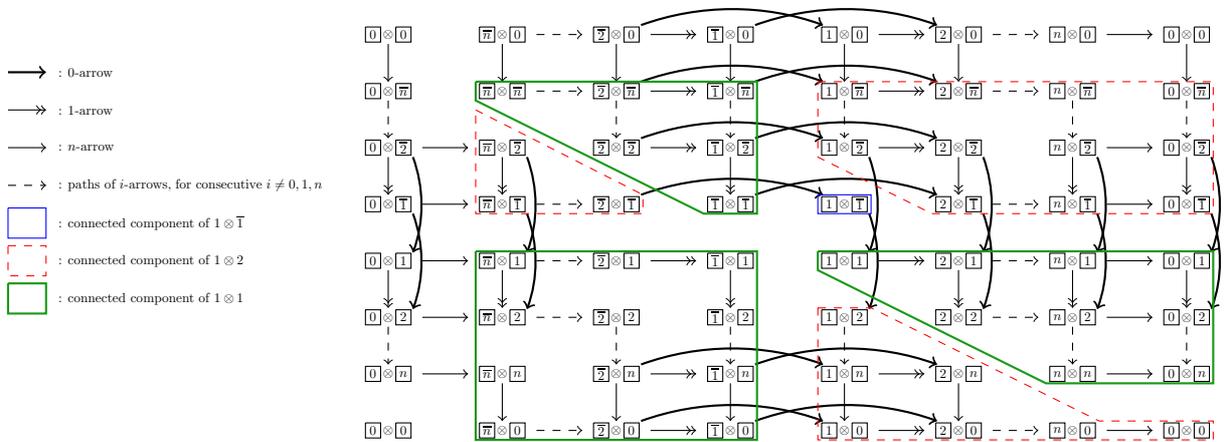

The energy $H$ on $\B\ot \B$ such that $H(0\ot 0)=0$ is given by the matrix
\begin{equation}\label{eq:energyB}
H= \bordermatrix{
\text{}&1&2&\cdots&n&0&\overline{n}&\cdots&\overline{2}&\overline{1}
\cr 1 &1&\cdots &\cdots&1&1&1&\cdots & \cdots &1
\cr 2& 0 &\ddots&1^*&\vdots&\vdots&\vdots&&&\vdots
\cr \vdots&\vdots&\ddots&\ddots &\vdots&\vdots&\vdots&1^*&&\vdots
\cr n&\vdots&&\ddots&1&1&1&\cdots&\cdots&1
\cr 0&\vdots&&&\ddots&0&1&\cdots&\cdots&1
\cr \overline{n}&\vdots&&&&\ddots&1&\cdots&\cdots&1
\cr \vdots&\vdots&&0^*&&&\ddots&\ddots&1^*&\vdots
\cr \overline{2}&0&0&&&& &\ddots &\ddots &\vdots
\cr \overline{1}&-1&0&\cdots&\cdots&\cdots&\cdots&\cdots& 0&1 
} .
\end{equation}

There are three standard modules of level $1$:
\begin{itemize}
\item $L(\Lambda_n)$, with ground state path $\p_{\Ll_n}=\cdots\,0\,0\,0\,0,$
\item $L(\Lambda_0)$, with ground state path $\p_{\Ll_0}=\cdots\,\overline{1}\,1\,\overline{1}\,1\,\overline{1},$
\item $L(\Lambda_1)$, with ground state path $\p_{\Ll_1}=\cdots\,1\,\overline{1}\,1\,\overline{1}\,1.$
\end{itemize}

The character formula for $L(\Lambda_n)$, which was already proved by the second author via other methods in \cite{Konan_ww2}, can be proved quite easily with Theorem \ref{theo:formchar} without any of the novelties introduced in this paper. Therefore we leave this proof to the interested reader.

However, $L(\Lambda_0)$ and $L(\Lambda_1)$ have non-constant ground state paths, so our multi-grounded partitions are once again useful to prove the character formulas of Theorem \ref{theo:charbn1}.

\subsubsection{Character for $L(\Lambda_0)$}
We start with $L(\Lambda_0)$. Note that the energy in \eqref{eq:energyB} is the exactly the same as the energy in \eqref{eq:Ha2}, except that the row and column $0$ were added.
Thus, we proceed exactly as we did for $A_{2n-1}^{(1)}(n\geq 3)$, except that we add parts coloured $c_0$. 

We apply again \Thm{theo:formchar2} with $d=2$ and $D=2$, which gives
\begin{equation}
\label{eq:thB0}
\sum_{\pi\in \, _2^2\Pp_{c_{\overline{1}}c_1}^{\gg}} C(\pi)q^{|\pi|} = \frac{e^{-\Ll}\mathrm{ch}(L(\Lambda_0))}{(q^2;q^2)_{\infty}},
\end{equation}
where $q=e^{-\delta/2}$ and $c_b=e^{\wt b}$ for all $b\in \B$.

If we temporarily forget the parity of the number of parts, the multi-grounded partitions of $_2^2\Pp_{c_{\overline{1}}c_1}$ in this section are obtained from the multi-grounded partitions of $_2^2\Pp_{c_{\overline{1}}c_1}$ from Section \ref{sec:A0} by adding odd parts coloured $c_0$ which can repeat (and placing them between $c_n$ and $c_{\overline{n}}$ in the partial order). So we obtain the generating function
$$\frac{(-c_1q,-c_{\overline{1}}q,\ldots,-c_nq,-c_{\overline{n}}q;q^2)_{\infty}}{(c_{\overline{1}}c_1q^4;q^4)_{\infty} (c_0q;q^2)_{\infty}} .$$
Now taking into account that the number of parts must be even, we find that  $_2^2\Pp_{c_{\overline{1}}c_1}^{\gg}$ is generated by
$$\sum_{\pi \in \, _2^2\Pp_{c_{\overline{1}}c_1}^{\gg}} C(\pi)q^{|\pi|}= \frac{1}{2(c_{\overline{1}}c_1q^4;q^4)_{\infty}}\left(\frac{(-c_1q,-c_{\overline{1}}q,\ldots,-c_nq,-c_{\overline{n}}q;q^2)_{\infty}}{(c_0q;q^2)_{\infty}}+\frac{(c_1q,c_{\overline{1}}q,\ldots,c_nq,c_{\overline{n}}q;q^2)_{\infty}}{(-c_0q;q^2)_{\infty}}\right). $$
By taking $c_0=c_{\overline{1}}c_1=1$, we then obtain
\begin{equation*}
 \sum_{\pi \in \, _2^2\Pp_{c_{\overline{1}}c_1}^{\gg}} C(\pi)q^{|\pi|}=\frac{(-q,-c_1q,-c_{\overline{1}}q,\ldots,-c_nq,-c_{\overline{n}}q;q^2)_{\infty}+(q,c_1q,c_{\overline{1}}q,\ldots,c_nq,c_{\overline{n}}q;q^2)_{\infty}}{2(q^2;q^2)_{\infty}},
\end{equation*}
and using \eqref{eq:thB0}, we obtain \eqref{eq:b0}.

\subsubsection{Character for $L(\Lambda_1)$}
To compute the character for $L(\Lambda_1)$, we do exactly the same reasoning as for $L(\Lambda_0)$: we start from the generating function corresponding to $L(\Lambda_1)$ in $A_{2n-1}^{(2)}$, and we add the parts coloured $c_0$.
We obtain
\begin{align*}
 \sum_{\pi \in \, _2^2\Pp_{c_1 c_{\overline{1}}}^{\gg}} C(\pi)q^{|\pi|}=\frac{1}{2(q^2;q^2)_{\infty}}\Big( &(-q,-c_1q^3,-c_{\overline{1}}q^{-1},-c_2q,-c_{\overline{2}}q,\ldots,-c_nq,-c_{\overline{n}}q;q^2)
 \\&+(q,c_1q^3,c_{\overline{1}}q^{-1},c_2q,c_{\overline{2}}q\ldots,c_nq,c_{\overline{n}}q;q^2)\Big),
\end{align*}
and \eqref{eq:b1} follows.

\subsection{The Lie algebra $D_{n}^{(1)}(n\geq 4)$}
We conclude this section of examples with the Lie algebra $D_{n}^{(1)}$ for $n\geq 4$.

The crystal $\B$ of the vector representation of $D_{n}^{(1)}(n\geq 4)$ is given by the crystal graph in Figure \ref{fig9}, with for all $u\in \{1,\ldots,n\}$
$$\wt (u) = -\wt (\overline{u}) = \frac{\alpha_n}{2}+\frac{\alpha_{n-1}}{2}+ \sum_{i=u}^{n-2} \alpha_i.$$
Here, the null root is
\[\delta = \alpha_0+\alpha_1+\alpha_{n-1}+\alpha_n+2\sum_{i=2}^{n-2} \alpha_i .\]

\begin{figure}[H]
\begin{center}
\begin{tikzpicture}[scale=0.8, every node/.style={scale=0.8}]
\draw (-4,0) node{$\B$ :};
\draw (-6,-0.75) node{$b^{\Ll_0}=b_{\Ll_1}=1\qquad b^{\Ll_1}=b_{\Ll_0}=\overline{1}$};
\draw (-6,-1.5) node{$\p_{\Ll_0}=(\cdots\,\overline{1}\,1\,\overline{1}\,1\,\overline{1})\qquad \p_{\Ll_1}=(\cdots\,1\,\overline{1}\,1\,\overline{1}\,1)$};
\draw (-6,-2.25) node{$b^{\Ll_n}=b_{\Ll_{n-1}}=\overline{n}\qquad b^{\Ll_{n-1}}=b_{\Ll_n}=n$};
\draw (-6,-3) node{$\p_{\Ll_{n-1}}=(\cdots\,\overline{n}\,n\,\overline{n}\,n\,\overline{n})\qquad \p_{\Ll_n}=(\cdots\,n\,\overline{n}\,n\,\overline{n}\,n)$};
\draw (0,0) node {$1$};
\draw (1.5,0) node {$2$};
\draw (5,0) node {$n-1$};
\draw (6.75,0) node {$n$};
\draw (3.15,0) node {$\cdots$};
\draw (0,-1.5) node {$\overline{1}$};
\draw (1.5,-1.5) node {$\overline{2}$};
\draw (5,-1.5) node {$\overline{n-1}$};
\draw (6.75,-1.5) node {$\overline{n}$};
\draw (3.15,-1.5) node {$\cdots$};

\foreach \x in {0,1,4.5} 
\foreach \y in {0,1}
\draw (-0.2+1.5*\x,0.2-1.5*\y)--(0.2+1.5*\x,0.2-1.5*\y)--(0.2+1.5*\x,-0.2-1.5*\y)--(-0.2+1.5*\x,-0.2-1.5*\y)--(-0.2+1.5*\x,0.2-1.5*\y);
\foreach \x in {0} 
\foreach \y  in {0,1}
\draw (-0.45+5+2*\x,0.2-1.5*\y)--(0.45+5+2*\x,0.2-1.5*\y)--(0.45+5+2*\x,-0.2-1.5*\y)--(-0.45+5+2*\x,-0.2-1.5*\y)--(-0.45+5+2*\x,0.2-1.5*\y);

\draw [thick, <-] (0.05,-0.25)--(1.45,-1.25);
\draw [thick, <-] (1.45,-0.25)--(0.05,-1.25);
\draw (0.3,-0.9) node {\footnotesize{$0$}};
\draw (1.2,-0.9) node {\footnotesize{$0$}};

\draw [thick, ->] (0.05+5.25,-0.25)--(1.45+5.25,-1.25);
\draw [thick, ->] (1.45+5.25,-0.25)--(0.05+5.25,-1.25);
\draw (0.3+5.25,-0.9) node {\footnotesize{$n$}};
\draw (1.2+5.25,-0.9) node {\footnotesize{$n$}};

\draw [thick, ->] (0.3,0)--(1.2,0);
\draw [thick, <-] (0.3,-1.5)--(1.2,-1.5);
\draw (0.75,0.15) node {\footnotesize{$1$}};
\draw (0.75,-1.67) node {\footnotesize{$1$}};
\draw [thick, ->] (1.8,0)--(2.7,0);
\draw [thick, <-] (1.8,-1.5)--(2.7,-1.5);
\draw (2.25,0.15) node {\footnotesize{$2$}};
\draw (2.25,-1.67) node {\footnotesize{$2$}};
\draw [thick, ->] (3.6,0)--(4.5,0);
\draw [thick, <-] (3.6,-1.5)--(4.5,-1.5);
\draw (4.05,0.15) node {\footnotesize{$n-2$}};
\draw (4.05,-1.67) node {\footnotesize{$n-2$}};
\draw [thick, ->] (5.6,0)--(6.4,0);
\draw [thick, <-] (5.6,-1.5)--(6.4,-1.5);
\draw (6,0.15) node {\footnotesize{$n-1$}};
\draw (6,-1.67) node {\footnotesize{$n-1$}};

\end{tikzpicture}
\end{center}
\caption{Crystal graph $\B$  of the vector representation for the Lie algebra $D_{n}^{(1)}(n\geq 4)$}
\label{fig9}
\end{figure}
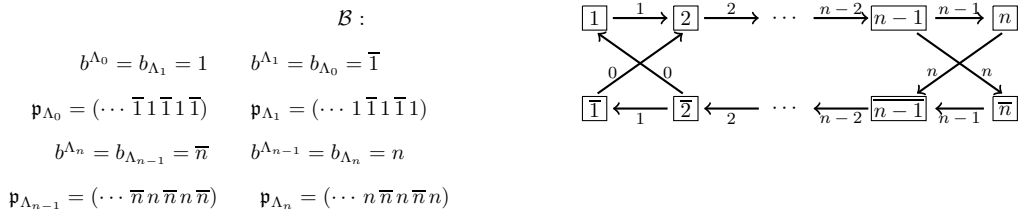

The crystal graph for $\B\ot \B$ is given in Figure \ref{fig10}, where we wrote $-1$ instead of $n-1$ for space reasons.
\begin{figure}[H]
\begin{center}
\begin{tikzpicture}[scale=0.5, every node/.style={scale=0.5}]
\draw [thick,->] (-10,0)--(-9,0);
\draw (-8.8,0) node[right] {: $0$-arrow};
\draw [->>] (-10,-1)--(-9,-1);
\draw (-8.8,-1) node[right] {: $1$-arrow};
\draw [->] (-10,-2)--(-9,-2);
\draw (-8.8,-2) node[right] {: $-1$-arrow};
\draw [thick,->>] (-10,-3)--(-9,-3);
\draw (-8.8,-3) node[right] {: $n$-arrow};
\draw [dashed,->] (-10,-4)--(-9,-4);
\draw (-8.8,-4) node[right] {: paths of $i$-arrows, for consecutive $i\neq 0,1,-1,n$};
\draw [blue] (-10,-4.6)--(-9,-4.6)--(-9,-5.4)--(-10,-5.4)--(-10,-4.6);
\draw (-8.8,-5) node[right] {: connected component of $1\ot \overline{1}$};
\draw [dashed,red] (-10,-5.6)--(-9,-5.6)--(-9,-6.4)--(-10,-6.4)--(-10,-5.6);
\draw (-8.8,-6) node[right] {: connected component of $1\ot 2$};
\draw [thick,foge] (-10,-6.6)--(-9,-6.6)--(-9,-7.4)--(-10,-7.4)--(-10,-6.6);
\draw (-8.8,-7) node[right] {: connected component of $1\ot 1$};
\draw (-10,-8) node[right] {$a\neq n,-1$};
\draw (-10,-9) node[right] {$b\neq \overline{n},\overline{-1}$};
\foreach \x in {0,...,7} 
\foreach \y in {0,...,7}
\draw (-0.2-0.4+3*\x,0.2-1.5*\y)--(0.2-0.4+3*\x,0.2-1.5*\y)--(0.2-0.4+3*\x,-0.2-1.5*\y)--(-0.2-0.4+3*\x,-0.2-1.5*\y)--(-0.2-0.4+3*\x,0.2-1.5*\y);
\foreach \x in {0,...,7} 
\foreach \y in {0,...,7}
\draw (-0.2+0.4+3*\x,0.2-1.5*\y)--(0.2+0.4+3*\x,0.2-1.5*\y)--(0.2+0.4+3*\x,-0.2-1.5*\y)--(-0.2+0.4+3*\x,-0.2-1.5*\y)--(-0.2+0.4+3*\x,0.2-1.5*\y);

\foreach \x in {0,...,7} 
\foreach \y in {0,...,7}
\draw (3*\x,-1.5*\y) node {$\ot$};

\foreach \x in {0,...,7} 
\draw (0.4+3*\x,0) node {$\overline{n}$};
\foreach \y in {0,...,7} 
\draw (-0.4,-1.5*\y) node {$\overline{n}$};
\foreach \x in {0,...,7} 
\draw (0.4+3*\x,-10.5) node {$n$};
\foreach \y in {0,...,7} 
\draw (-0.4+21,-1.5*\y) node {$n$};

\foreach \x in {0,...,7}
\draw (0.4+3*\x,-1.5) node {$\overline{-1}$};
\foreach \y in {0,...,7} 
\draw (-0.4+3,-1.5*\y) node {$\overline{-1}$};

\foreach \x in {0,...,7}
\draw (0.4+3*\x,-3) node {$\overline{2}$};
\foreach \y in {0,...,7} 
\draw (-0.4+6,-1.5*\y) node {$\overline{2}$};

\foreach \x in {0,...,7}
\draw (0.4+3*\x,-4.5) node {$\overline{1}$};
\foreach \y in {0,...,7} 
\draw (-0.4+9,-1.5*\y) node {$\overline{1}$};

\foreach \x in {0,...,7}
\draw (0.4+3*\x,-6) node {$1$};
\foreach \y in {0,...,7} 
\draw (-0.4+12,-1.5*\y) node {$1$};

\foreach \x in {0,...,7}
\draw (0.4+3*\x,-7.5) node {$2$};
\foreach \y in {0,...,7} 
\draw (-0.4+15,-1.5*\y) node {$2$};

\foreach \x in {0,...,7}
\draw (0.4+3*\x,-9) node {$-1$};
\foreach \y in {0,...,7} 
\draw (-0.4+18,-1.5*\y) node {$-1$};

\foreach \x in {1,...,5,7} 
\draw [->] (3*\x,-0.25)--(3*\x,-1.25);
\foreach \x in {1,...,5,7} 
\draw [->] (3*\x,-9.25)--(3*\x,-10.25);

\foreach \x in {0,2,3,4,6,7} 
\draw [dashed,->] (3*\x,-1.75)--(3*\x,-2.75);
\foreach \x in {0,2,3,4,6,7} 
\draw [dashed,->] (3*\x,-7.75)--(3*\x,-8.75);

\foreach \x in {2,4}
\foreach \y in {0,1,2,4,6,7}
\draw [->>] (0.9+3*\x,-1.5*\y)--(2.1+3*\x,-1.5*\y);

\foreach \x in {2,3}
\foreach \y in {0,1,2,3,6,7}
\draw [thick,->] (0.65+3*\x,0.25-1.5*\y) arc (110:70:7);

\foreach \y in {2,4}
\foreach \x in {0,1,3,5,6,7}
\draw [->>] (3*\x,-0.25-1.5*\y)--(3*\x,-1.25-1.5*\y);

\foreach \y in {2,3}
\foreach \x in {0,1,4,5,6,7}
\draw [thick,->] (0.65+3*\x,-0.25-1.5*\y) arc (20:-20:3.7);

\foreach \y in {0,2,3,4,5,6} 
\draw [->] (0.9,-1.5*\y)--(2.1,-1.5*\y);
\foreach \y in {0,2,3,4,5,6} 
\draw [->] (18.9,-1.5*\y)--(20.1,-1.5*\y);

\foreach \y in {0,1,3,4,5,7} 
\draw [dashed,->] (3.9,-1.5*\y)--(5.1,-1.5*\y);
\foreach \y in {0,1,3,4,5,7} 
\draw [dashed,->] (15.9,-1.5*\y)--(17.1,-1.5*\y);

\draw [blue] (-0.7+12,0.25-4.5)--(0.7+12,0.25-4.5)--(0.7+12,-0.25-4.5)--(-0.7+12,-0.25-4.5)--(-0.7+12,0.25-4.5);

\draw [dashed,red] (-0.7+12,0.25)--(0.7+21,0.25)--(0.7+21,-0.25-4.5)--(-0.7+15,-0.25-4.5)--(-0.7+12,-0.25-3)--(-0.7+12,0.25);
\draw [dashed,red] (-0.7,0.25-1.5)--(0.7,0.25-1.5)--(0.7+6,0.25-4.5)--(0.7+6,-0.25-4.5)--(-0.7,-0.25-4.5)--(-0.7,0.25-1.5);
\draw [dashed,red] (-0.7+12,0.25-7.5)--(0.7+12,0.25-7.5)--(0.7+18,0.25-10.5)--(0.7+18,-0.25-10.5)--(-0.7+12,-0.25-10.5)--(-0.7+12,0.25-7.5);

\draw [dashed,red] (-0.7,0.25-10.5)--(0.7,0.25-10.5)--(0.7,-0.25-10.5)--(-0.7,-0.25-10.5)--(-0.7,0.25-10.5);

\draw [thick,foge] (-0.7,0.25-6)--(0.7+9,0.25-6)--(0.7+9,-0.25-10.5)--(-0.7+3,-0.25-10.5)--(-0.7+3,-0.25-9)--(-0.7,-0.25-9)--(-0.7,0.25-6);
\draw [thick,foge] (-0.7,0.25)--(0.7+9,0.25)--(0.7+9,-0.25-4.5)--(-0.7+9,-0.25-4.5)--(-0.7,-0.25)--(-0.7,0.25);
\draw [thick,foge] (-0.7+12,0.25-6)--(0.7+21,0.25-6)--(0.7+21,-0.25-10.5)--(-0.7+21,-0.25-10.5)--(-0.7+12,-0.25-6)--(-0.7+12,0.25-6);

\foreach \x in {1,2,5,6} 
\foreach \y in {8,9}
\draw (-0.2-0.4+3*\x,0.2-1.5*\y)--(0.2-0.4+3*\x,0.2-1.5*\y)--(0.2-0.4+3*\x,-0.2-1.5*\y)--(-0.2-0.4+3*\x,-0.2-1.5*\y)--(-0.2-0.4+3*\x,0.2-1.5*\y);
\foreach \x in {1,2,5,6} 
\foreach \y in {8,9}
\draw (-0.2+0.4+3*\x,0.2-1.5*\y)--(0.2+0.4+3*\x,0.2-1.5*\y)--(0.2+0.4+3*\x,-0.2-1.5*\y)--(-0.2+0.4+3*\x,-0.2-1.5*\y)--(-0.2+0.4+3*\x,0.2-1.5*\y);

\foreach \x in {1,2,5,6} 
\foreach \y in {8,9}
\draw (3*\x,-1.5*\y) node {$\ot$};

\foreach \x in {1,5} 
\foreach \y in {8,9}
\draw [thick,->>] (0.9+3*\x,-1.5*\y)--(2.1+3*\x,-1.5*\y);

\draw (0.4+3,-12) node {$n$};
\draw (0.4+6,-12) node {$\overline{-1}$};
\draw (0.4+15,-12) node {$-1$};
\draw (0.4+18,-12) node {$\overline{n}$};

\draw (-0.4+3,-13.5) node {$n$};
\draw (-0.4+6,-13.5) node {$\overline{-1}$};
\draw (-0.4+15,-13.5) node {$-1$};
\draw (-0.4+18,-13.5) node {$\overline{n}$};

\foreach \x in {1,2,5,6} 
\draw (-0.4+3*\x,-12) node {$a$};

\foreach \x in {1,2,5,6} 
\draw (0.4+3*\x,-13.5) node {$b$};

\end{tikzpicture}
\end{center}
\caption{Crystal graph of $\B\ot \B$ for the Lie algebra $D_{n}^{(1)} (n\geq 4)$}
\label{fig10}
\end{figure}
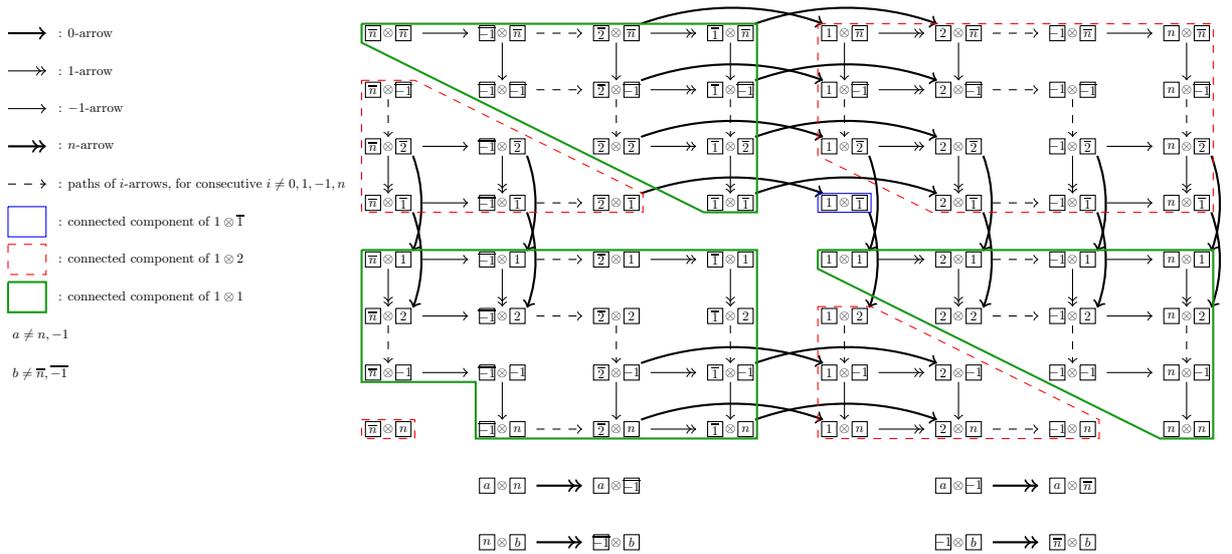

The energy $H$ on $\B\ot \B$ such that $H(\overline{n}\ot n)=0$ is given by the matrix
\begin{equation*}\label{eq:energyD1}
H= \bordermatrix{
\text{}&1&2&\cdots&n-1&n&\overline{n}&\overline{n-1}&\cdots&\overline{2}&\overline{1}
\cr 1 &1&\cdots &\cdots&\cdots&1&1&\cdots & \cdots & \cdots&1
\cr 2& 0 &\ddots&&1^*&\vdots&\vdots&&&&\vdots
\cr \vdots&\vdots&\ddots&\ddots &&\vdots&\vdots&&1^*&&\vdots
\cr n-1&\vdots&&\ddots&\ddots&\vdots&1&\cdots&\cdots&\cdots&1
\cr n&\vdots&&&\ddots&1&0&1&\cdots&\cdots&1
\cr \overline{n}&\vdots&&&&0&1&\cdots&\cdots&\cdots&1
\cr \overline{n-1}&\vdots&&0^*&&&\ddots&\ddots&1^*&&1
\cr \vdots&\vdots&&&&&&\ddots&\ddots&&\vdots
\cr \overline{2}&0&0&&&&& &\ddots &\ddots &\vdots
\cr \overline{1}&-1&0&\cdots&\cdots&\cdots&\cdots&\cdots&\cdots& 0&1 
} .
\end{equation*}
Note that this is almost the same as the energy matrix \eqref{eq:Ha2}, except that here, we have $H(\overline{n}\ot n)=0$ instead of $1$.

These difference conditions correspond the partial order 
\begin{equation*}
\label{eq:po}
\cdots\ll 0_{c_{n-1}}\ll\begin{array}{c} 0_{c_n} \\ 0_{c_{\overline{n}}}\end{array}\ll 0_{c_{\overline{n-1}}}\ll \cdots \ll 0_{c_{\overline{2}}}\ll\begin{array}{c} 0_{c_{\overline{1}}} \\ 1_{c_1}\end{array}\ll 1_{c_2}\ll\cdots\ll 1_{c_{n-1}}\ll\cdots.
\end{equation*}

There are four standard modules of level $1$:
\begin{itemize}
\item $L(\Lambda_0)$, with ground state path $\p_{\Ll_0}=\cdots\,\overline{1}\,1\,\overline{1}\,1\,\overline{1},$
\item $L(\Lambda_1)$, with ground state path $\p_{\Ll_1}=\cdots\,1\,\overline{1}\,1\,\overline{1}\,1,$
\item $L(\Lambda_{n-1})$, with ground state path $\p_{\Ll_{n-1}}=\cdots\,\overline{n}\,n\,\overline{n}\,n\,\overline{n},$
\item $L(\Lambda_n)$, with ground state path $\p_{\Ll_n}=\cdots\,n\,\overline{n}\,n\,\overline{n}\,n,$
\end{itemize}
We have 
\[H(\overline{1}\ot 1)=-H(1\ot \overline{1})=1\]
and 
\[H(\overline{n}\ot n)=H(n\ot \overline{n})=0\, ,\]
so the sums of the energies are $0$ on all these ground state paths, and we can choose $H_{\Lambda}=H$ for all the above-mentioned modules.

We show briefly how to apply Theorem \ref{theo:formchar2} to obtain Theorem \ref{theo:chardn1}. The principle is the same as in the previous sections.

\subsubsection{Character for $L(\Lambda_0)$}
As in the case of $A_{2n-1}^{(2)}$, we apply \Thm{theo:formchar2} with $d=2$ and $D=2$. We obtain
\begin{equation}
\label{eq:thd10}
\sum_{\pi\in \, _2^2\Pp_{c_{\overline{1}}c_1}^{\gg}} C(\pi)q^{|\pi|} = \frac{e^{-\Ll_0}\mathrm{ch}(L(\Lambda_0))}{(q^2;q^2)_{\infty}},
\end{equation}
where $q=e^{-\delta/2}$ and $c_b=e^{\wt b}$ for all $b\in \B$.

Again, the difference between consecutive parts of the multi-grounded partitions in $_2^2\Pp_{c_{\overline{1}}c_1}^{\gg}$ is even, and because $u^{(0)}=-1$ and $u^{(1)}=1$ again, all the parts are odd. So the partial order becomes
\[\begin{array}{c} (-1)_{c_{\overline{1}}} \\ 1_{c_1}\end{array}\ll 1_{c_2}\ll\cdots\ll 1_{c_{n-1}}\ll\begin{array}{c} 1_{c_n} \\ 1_{c_{\overline{n}}}\end{array}\ll 1_{c_{\overline{n-1}}}\ll \cdots \ll 1_{c_{\overline{2}}}\ll\begin{array}{c} 1_{c_{\overline{1}}} \\ 3_{c_1}\end{array}\ll 3_{c_2}\ll\cdots\ll 3_{c_{n-1}}\ll\cdots .\]

Here, in addition to the alternating sequences
 $$\cdots \ll (2k-1)_{c_{\overline{1}}}\ll (2k+1)_{c_1}\ll(2k-1)_{c_{\overline{1}}}\ll \cdots $$
already present in $A_{2n-1}^{(2)}$,
we also have to consider alternating sequences of the form 
$$ \cdots\ll(2k+1)_{c_n}\ll (2k+1)_{c_{\overline{n}}}\ll (2k+1)_{c_n}\cdots .$$
Thus the generating function without the condition on the parity of the number of parts is
$$\frac{(-c_1q,-c_{\overline{1}}q,\ldots,-c_nq,-c_{\overline{n}}q;q^2)_{\infty}}{(c_{\overline{1}}c_1q^4;q^4)_{\infty}(c_nc_{\overline{n}}q^2;q^4)_{\infty}} .$$

So we deduce
\begin{align*}
 \sum_{\pi \in _2^2\Pp_{c_{\overline{1}}c_1}^{\gg}} C(\pi)q^{|\pi|}= \frac{1}{2(c_{\overline{1}}c_1q^4;q^4)_{\infty}(c_nc_{\overline{n}}q^2;q^4)_{\infty}} \Big( &(-c_1q,-c_{\overline{1}}q,\ldots,-c_nq,-c_{\overline{n}}q;q^2)_{\infty}
 \\&+(c_1q,c_{\overline{1}}q,\ldots,c_nq,c_{\overline{n}}q;q^2)_{\infty}\Big).
\end{align*}

By taking $c_nc_{\overline{n}}=c_1 c_{\overline{1}}=1$, we obtain
\[
 \sum_{\pi \in _2^2\Pp_{c_{\overline{1}}c_1}^{\gg}} C(\pi)q^{|\pi|}= \frac{1}{2(q^2;q^2)_{\infty}}\left((-c_1q,-c_{\overline{1}}q,\ldots,-c_nq,-c_{\overline{n}}q;q^2)_{\infty}+(c_1q,c_{\overline{1}}q,\ldots,c_nq,c_{\overline{n}}q;q^2)_{\infty}\right),
\]
and using \eqref{eq:thd10}, we deduce \eqref{eq:d10}.

\subsubsection{Character for $L(\Lambda_1)$}
This case works in the exact same way as the previous one, so we omit the details. We obtain
\begin{align*}
 \sum_{\pi \in\, _2^2\Pp_{c_1c_{\overline{1}}}^{\gg}} C(\pi)q^{|\pi|} =\frac{1}{2(q^2;q^2)_{\infty}} \Big(&(-c_1q^3,-c_{\overline{1}}q^{-1},-c_2q,-c_{\overline{2}}q,\ldots,-c_nq,-c_{\overline{n}}q;q^2)_{\infty}
 \\&+(c_1q^3,c_{\overline{1}}q^{-1},c_2q,c_{\overline{2}}q\ldots,c_nq,c_{\overline{n}}q;q^2)_{\infty}\Big),
\end{align*}
and we conclude with Theorem \ref{theo:formchar2} as usual.

\subsubsection{Character for $L(\Lambda_{n-1})$}
Since $H(\overline{n}\ot n)=H(n\ot \overline{n})=0$, we can choose $D=1$, and we have $u^{(0)}=u^{(1)}=0$. 

We apply Theorem \ref{theo:formchar2} with $D=d=1$, and obtain
\begin{equation}
\label{eq:thd1n-1}
\sum_{\pi \in _2\Pp_{c_{\overline{n}}c_n}^{\gg}} C(\pi)q^{|\pi|} =  \frac{e^{-\Ll_{n-1}}\mathrm{ch}(L(\Lambda_{n-1}))}{(q;q)_{\infty}},
\end{equation}
where $q=e^{-\delta}$ and $c_b=e^{\wt b}$ for all $b\in \B$.

Now $d=1$ so we consider directly the partial order \eqref{eq:po}.
By reasoning on the tail $(0_{c_{\overline{n}}},0_{c_n})$ of the multi-grounded partitions in $_2\Pp_{c_{\overline{n}}c_n}^{\gg}$ in the same way as in the case of $A_{2n-1}^{(2)}$, and using \eqref{eq:evennumberparts} again, we obtain the generating function:
\begin{align*}
 \sum_{\pi \in _2\Pp_{c_{\overline{n}}c_n}^{\gg}} C(\pi)q^{|\pi|}=\frac{1}{2(c_1c_{\overline{1}}q;q^2)_{\infty}(c_nc_{\overline{n}}q^2;q^2)_{\infty}}\Big(&(-c_1q,-c_{\overline{1}},\ldots,-c_{n-1}q,-c_{\overline{n-1}},-c_n,-c_{\overline{n}}q;q)_{\infty}\\
 &+(c_1q,c_{\overline{1}},\ldots,c_{n-1}q,c_{\overline{n-1}},c_n,c_{\overline{n}}q;q)_{\infty}\Big),
\end{align*}
and using \eqref{eq:thd1n-1}, we deduce \eqref{eq:d1n-1} (note that in Theorem \ref{theo:chardn1}, we have set $q=e^{-\delta/2}$ for the whole theorem for consistency, so the $q$'s of this formula are squared).

\subsubsection{Character for $L(\Lambda_n)$}
We do the same reasoning as before except that now the tail is $(0_{c_n},0_{c_{\overline{n}}})$, and we obtain
\begin{align*}
 \sum_{\pi \in _2\Pp_{c_nc_{\overline{n}}}^{\gg}} C(\pi)q^{|\pi|}=\frac{1}{2(c_1c_{\overline{1}}q;q^2)_{\infty}(c_nc_{\overline{n}}q^2;q^2)_{\infty}}\Big(&(-c_1q,-c_{\overline{1}},\ldots,-c_{n-1}q,-c_{\overline{n-1}},-c_nq,-c_{\overline{n}};q)_{\infty}\\
 &+(c_1q,c_{\overline{1}},\ldots,c_{n-1}q,c_{\overline{n-1}},c_nq,c_{\overline{n}};q)_{\infty}\Big),
\end{align*}
and we conclude once again with Theorem \ref{theo:formchar2}.

\section{Conclusion}
The point of this paper is to introduce the notion of multi-grounded partitions and to show how they can be used to obtain character formulas, even for modules whose ground state paths are not constant. As examples, we studied the level $1$ standard modules of several classical affine Lie algebras which have relatively simple energy functions. However, our method can be applied for representations at any level, which we plan to do in subsequent papers. 

\section*{Acknowledgements}
We thank Leonard Hardiman and Ole Warnaar for their comments on earlier versions of this paper and their helpful suggestions to improve it.

\bibliographystyle{alpha}
\bibliography{biblio}

\end{document}